\documentclass[a4paper,12pt]{article}

\usepackage{slashbox}
\usepackage{authblk}
\usepackage{mathtools}
\usepackage{tikz}
\usetikzlibrary{calc}
\usepackage[T1]{fontenc}
\usepackage[utf8]{inputenc}

\usepackage{amssymb,amsfonts, amsmath,amsthm,epsfig,tabularx}

\usepackage{enumitem}
\usepackage{etoolbox}

\AtBeginEnvironment{proof}{\setcounter{claim}{0}}

\usepackage{float}
\restylefloat{table}

\usepackage[hyperfootnotes=true,colorlinks=true]{hyperref}
\hypersetup{
    colorlinks=true,
    linkcolor=blue,
    citecolor=magenta,      
    urlcolor=black
}

\usepackage{caption} 
\usepackage{subcaption} 

\newtheorem{theorem}{Theorem}

\newtheorem{lemma}[theorem]{Lemma}

\newtheorem*{observationx}{Observation}
\newtheorem{definition}[theorem]{Definition}

\newtheorem{example}[theorem]{Example}
\newtheorem{conjecture}[theorem]{Conjecture}

\newtheorem{problem}[theorem]{Problem}

\DeclareMathOperator{\SmallGroup}{SmallGroup}
\DeclareMathOperator{\Group}{Group}
\DeclareMathOperator{\CVT}{CVT}
\DeclareMathOperator{\Tr}{Tr}

\title{On regular graphs with Šolt{\'e}s vertices}

\author[1,2,3]{Nino Bašić}
\author[4]{Martin Knor}
\author[1,5,6]{Riste \v{S}krekovski}
 
\affil[1]{{\small FAMNIT, University of Primorska, Koper, Slovenia}}
\affil[2]{{\small IAM, University of Primorska, Koper, Slovenia}}
\affil[3]{{\small Institute of Mathematics, Physics and Mechanics, Ljubljana, Slovenia}}
\affil[4]{{\small Slovak University of Technology in Bratislava, Slovakia}}
\affil[5]{{\small Faculty of Mathematics and Physics, University of Ljubljana, Slovenia}}
\affil[6]{{\small Faculty of Information Studies, Novo mesto, Slovenia}}

\begin{document}

\maketitle

\begin{abstract}
Let $W(G)$ be the Wiener index of a graph $G$. We say that a vertex $v \in V(G)$ is a \emph{Šolt{\'e}s vertex} in $G$ if $W(G - v) = W(G)$,
i.e.\ the Wiener index does not change if the vertex $v$ is removed.
In 1991, Šolt{\'e}s posed the problem of identifying all connected graphs $G$ with the property that all vertices of $G$ are Šolt{\'e}s vertices.
The only such graph known to this day is $C_{11}$. As the original problem appears to be too challenging, several relaxations 
were studied: one may look for graphs with at least $k$ Šoltes vertices; or
one may look for \emph{$\alpha$-Šolt{\'e}s graphs}, i.e.\ graphs where the ratio between the number of Šolt{\'e}s
vertices and the order of the graph is at least $\alpha$. Note that the original problem is, in fact, to find all $1$-Šolt{\'e}s graphs.
We intuitively believe that every $1$-Šolt{\'e}s graph has to be regular and
has to possess a high degree of symmetry. Therefore, 
we are interested in \emph{regular} graphs that contain one or more Šolt{\'e}s vertices.
In this paper, we present several partial results.
For every $r\ge 1$ we describe a construction of an infinite family of cubic $2$-connected graphs with at least $2^r$ Šolt{\'e}s vertices.
Moreover, we report that a computer search on publicly available collections of vertex-transitive graphs did not reveal any $1$-Šolt{\'e}s graph.
We are only able to provide examples of large $\frac{1}{3}$-Šolt{\'e}s graphs that are obtained by truncating certain cubic vertex-transitive graphs.
This leads us to believe that no $1$-Šolt{\'e}s graph other than $C_{11}$ exists.

\vspace{0.5\baselineskip}

\vspace{0.5\baselineskip}
\noindent
\textbf{Math.\ Subj.\ Class.\ (2020):} 05C12, 05C90, 20B25
\end{abstract}
%
%
\section{Introduction}

All graphs under consideration in this paper are simple and undirected. The \emph{Wiener index} of a graph $G$, denoted by $W(G)$, is defined as
\begin{equation}
W(G) = \frac{1}{2} \sum_{u \in V(G)} \sum_{v \in V(G)} d_G(u, v) = \sum_{\{u,v\} \subseteq V(G)} d_G(u, v),
\label{eq:1}
\end{equation}
where $d_G(u, v)$ is the distance between vertices $u$ and $v$ (i.e.\ the length of a shortest path between $u$ and $v$).
If the graph $G$ is disconnected, we may take $W(G) = \infty$.
The Wiener index was introduced in 1947 \cite{wiener47} and has been extensively studied ever since. For a recent survey on
the Wiener index see \cite{AMCsurvey}. The \emph{transmission} of a vertex $v$ in a graph $G$, denoted by $w_G(v)$, is defined as
$w_G(v) = \sum_{u  \in V(G)} d_G(u, v)$. Note that \eqref{eq:1} can also be expressed as $W(G) = \frac{1}{2} \sum_{u \in V(G)} w_G(u)$.

Let $v \in V(G)$. The graph obtaned from $G$ by removing a vertex $v$ is denoted by $G - v$. If we remove a vertex
$v$ from a graph $G$, any of the following scenarios may occur:
\begin{enumerate}[label=(\alph*)]
\item the Wiener index \emph{decreases} (e.g.\ $W(K_7) = 21$ and $W(K_7 - v) = W(K_6) = 15$);
\item the Wiener index \emph{increases} (e.g.\ $W(\mathrm{Wh}_{9}) = 56$ and $W(\mathrm{Wh}_{9} - v_0) = W(C_8) = 64$, where $\mathrm{Wh}_n$ denotes the
wheel graph on $n$ vertices and $v_0$ is the central vertex of the wheel);
\item the Wiener index \emph{does not change} (e.g.\ $W(\mathrm{Wh}_{8}) = 42$ and $W( \mathrm{Wh}_{8} - v_0) = W(C_7) = 42$, where $v_0$ is the central vertex
of $\mathrm{Wh}_{8}$).
\end{enumerate}
We say that a vertex $v \in V(G)$ is a \emph{Šolt{\'e}s vertex of $G$} if $W(G) = W(G - v)$, i.e.\ the Wiener index of $G$ does not change if the vertex $v$ is removed.
If $G$ is disconnected, with two non-trivial components or with at least three components, then every vertex is a Šolt{\'e}s vertex. Therefore, it is natural to require that $G$ is connected. 
Let
$$
S(G) = \{ v \in V(G) \mid W(G) = W(G - v) \}
$$
and let $0 \leq \alpha \leq 1$.
We say that a graph $G$ is an \emph{$\alpha$-Šolt{\'e}s graph} if $|S(G)| \geq \alpha |V(G)|$, i.e.\ the ratio between the
number of Šolt{\'e}s vertices of $G$ and the order of $G$ is at least $\alpha$. 
For example, the graph in Figure~\ref{fig:smallest_cubic_third} is the smallest cubic $\frac{1}{3}$-Šolt{\'e}s graph.
Note that $G$ is an $1$-Šolt{\'e}s graph if every vertex in $G$
is a Šolt{\'e}s vertex. The only $1$-Šolt{\'e}s graph known to this day is the cycle on $11$ vertices, $C_{11}$. In this paper, Šolt{\'e}s graph is simply the
synonym for $1$-Šolt{\'e}s graph.

\begin{figure}[!b]
\centering
\begin{tikzpicture}[scale=1.5,rotate=90]
\tikzstyle{edge}=[draw,thick]
\tikzstyle{every node}=[draw, circle, fill=blue!50!white, inner sep=1.5pt]
\node[fill=red!70!white] (v4) at (-1,-0.0) {};
\node[fill=red!70!white] (v10) at (1,-0.0) {};
\node[fill=red!70!white] (v8) at (-1,1) {};
\node[fill=red!70!white] (v11) at (1,1) {};
\node (v9) at (-0.6,-1.6) {}; 
\node (v7) at (0.6,-1.6) {};
\node[fill=red!70!white] (v16) at (-0.6,-0.6) {};
\node[fill=red!70!white] (v17) at (0.6,-0.6) {};
\node (v18) at (-0.6,-1) {};
\node (v19) at (0.6,-1) {};
\path[edge] (v18) -- (v19);
\path[edge] (v10) -- (v17) -- (v19) -- (v7);
\path[edge] (v4) -- (v16) -- (v18) -- (v9);
\path[edge] (v16) -- (v17);
\node (v0) at (0,-1.3) {};
\node (v5) at (0,-1.9) {};
\path[edge] (v9) -- (v5) -- (v7); 
\path[edge] (v0) -- (v5); 
\path[edge] (v9) -- (v0) -- (v7); 
\node[fill=red!70!white] (v20) at (-0.6,0.6+1) {};
\node[fill=red!70!white] (v21) at (0.6,0.6+1) {};
\node (v22) at (-0.6,1+1) {};
\node (v23) at (0.6,1+1) {};
\node (v3) at (-0.6,{1.6+1}) {};
\node (v6) at (0.6,{1.6+1}) {};
\node (v2) at (0,{1.3+1}) {};
\node (v1) at (0,{1.9+1}) {};
\path[edge] (v1) -- (v2); 
\path[edge] (v3) -- (v1) -- (v6) -- (v2) -- (v3); 
\node (v12) at (-0.3,0.8) {};
\node (v13) at (0.3,0.8) {};
\node (v14) at (-0.3,0.2) {};
\node (v15) at (0.3,0.2) {};
\path[edge] (v3) -- (v22) -- (v20) --  (v8) -- (v4); 
\path[edge] (v6) -- (v23) -- (v21) -- (v11) -- (v10); 
\path[edge] (v12) -- (v13) -- (v15) -- (v14) -- (v12); 
\path[edge] (v12) -- (v8);
\path[edge] (v14) -- (v4); 
\path[edge] (v13) -- (v11);
\path[edge] (v15) -- (v10); 
\path[edge] (v22) -- (v23); 
\path[edge] (v20) -- (v21); 
\end{tikzpicture}
\caption{The smallest cubic $\frac{1}{3}$-Šolt{\'e}s graph has $24$ vertices. Its Šolt{\'e}s vertices are coloured red.}
\label{fig:smallest_cubic_third}
\end{figure}
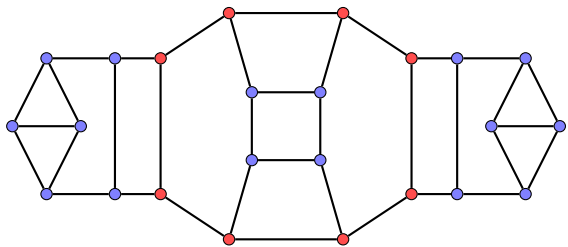

The Šolt{\'e}s problem \cite{soltes} was forgotten for nearly three decades. It was revived and popularised in 2018 by Knor~\emph{et al.}~\cite{Majst-1}.
They considered a relaxation of the original problem: one may look for graphs with
a prescribed number of Šolt{\'e}s vertices. They
showed that there exists a unicyclic graph on $n$ vertices with \emph{at least one} Šolt{\'e}s vertex for every $n \geq 9$. They also showed that there exists a unicyclic
graph with a cycle of length $c$ and at least one Šolt{\'e}s vertex for every $c \geq 5$, and that every graph is an induced subgraph of some larger graph with a Šolt{\'e}s vertex.
They have further shown that a Šolt{\'e}s vertex in a graph may have a prescribed degree~\cite{Majst-2}. Namely, they proved that for any $d \geq 3$
there exist infinitely many graphs with a Šolt{\'e}s vertex of degree $d$. Necessary conditions for the existence of Šolt{\'e}s vertices
in Cartesian products of graphs were also considered~\cite{Majst-2}.
In 2021, Bok~\emph{et al.}~\cite{Bok} showed that for every $k \geq 1$ there exist infinitely many \emph{cactus graphs} with exactly $k$ distinct Šolt{\'e}s vertices.
In 2021, Hu \emph{et al.}~\cite{Hu2021} studied a variation of the problem and showed that there exist infinitely many graphs where the Wiener index 
remains the same even if $r \geq 2$ distinct vertices are removed from the graph.

Akhmejanova \emph{et al.}~\cite{akh} considered another possible relaxation of the problem: Do there exist graphs with a given percentage of Šolt{\'e}s
vertices?
They constructed two infinite families of graphs with a relatively high proportion of Šolt{\'e}s vertices.
Their first family comprises graphs $B(k)$, $k \geq 2$, where $B(k)$ is a $\frac{2k}{5k+6}$-Šolt{\'e}s graph on $5k+6$ vertices.
Two vertices of $B(k)$ are of degree $k + 1$, while the remaining vertices are of degree $2$. The percentage of Šolt{\'e}s vertices is below $\frac{2}{5}$,
but tends to $\frac{2}{5}$ as $k$ goes to infinity. They also introduced a two-parametric infinite family $L(k, m)$, $m \geq 7$ and $k \geq \frac{m-3}{m-6}$.
Here, the percentage of Šolt{\'e}s vertices is below $\frac{1}{2}$, but tends to $\frac{1}{2}$ as $k$ goes to infinity for a fixed $m$. These graphs
contain at least one leaf, at least $km$ vertices of degree $2$, and a vertex of degree $km + 1$.

In the present paper we focus on \emph{regular graphs}.
Our intuition lead us to believe that the solutions to the original Šolt{\'e}s problem should be graphs having all vertices of the same degree.

\begin{conjecture}
If $G$ is a Šolt{\'e}s graph, then $G$ is regular.
\end{conjecture}

For a general regular graph $G$, the values $W(G - u)$ and $W(G - v)$ might be significantly different
for two different vertices $u, v \in V(G)$; it may happen that removal of one vertex increases the Wiener index,
while removal of the other vertex descreases the Wiener index.
However, $W(G - u)$ and $W(G - v)$ are equal if vertices
$u$ and $v$ belong to the same vertex orbit. Therefore, we believe that a Šolt{\'e}s graph
is likely to be vertex transitive.

\begin{conjecture}
If $G$ is a Šolt{\'e}s graph, then $G$ is vertex transitive.
\end{conjecture}

Among truncations of cubic vertex-transitive graphs we found several $\frac{1}{3}$-Šolt{\'e}s graphs; see Section~\ref{sec:concluding}.
Interestingly, all our examples are in fact Cayley graphs and this leads us to pose the following conjecture.

\begin{conjecture}
If $G$ is a Šolt{\'e}s graph, then $G$ is a Cayley graph.
\end{conjecture}

It is not hard to obtain small examples of regular graphs with Šolt{\'e}s vertices.
We used the \texttt{geng} \cite{nauty} software to generate small $k$-regular graphs (for $k = 3, 4$ and $5$).
Let $\mathcal{R}^r$ denote the class of all $r$-regular graphs and let $\mathcal{R}^r_n$ denote the set of $r$-regular graphs on $n$ vertices. Let $N(\mathcal{G}, k)$ be the number of graphs in the class $\mathcal{G}$ that contain exactly $k$ Šolt{\'e}s vertices.

Table~\ref{tbl:1} shows the numbers of (non-isomorphic) cubic graphs of orders $n \leq 24$ that contain Šolt{\'e}s vertices.
We can see that cubic graphs of order $12$ or less do not contain Šolt{\'e}s vertices. There are plenty of examples with one
Šolt{\'e}s vertex. Cubic graphs with two Šolt{\'e}s vertices first appear at order $n = 14$; there are three such graphs (see Figure~\ref{fig:small_examples}(a)--(c)).
Examples with three and four Šolt{\'e}s vertices appear at order $n = 16$; there is one cubic graph with three and two cubic graphs with four Šolt{\'e}s vertices
(see Figure~\ref{fig:small_examples}(d)--(f)).
At order $n = 18$, there are no graphs with three Šolt{\'e}s vertices, however there is only one graph with four Šolt{\'e}s vertices (see Figure~\ref{fig:small_examples}(g)).
Numbers of $4$-regular and $5$-regular graphs with respect to their number of Šolt{\'e}s vertices are given in Tables~\ref{tbl:2} and~\ref{tbl:3}, respectively.

\begin{table}[!htb]
\centering
\caption{The numbers of (non-isomorphic) cubic graphs with Šolt{\'e}s vertices.
There are no graphs with Šolt{\'e}s vertices for orders up to $12$.
The column labeled $|\mathcal{R}^3_n|$ gives the total number of cubic graphs of order $n$.
Symbol `-' is a replacement for $0$ (i.e.\ no such graph exists).
Each of the next columns gives the numbers of graphs with $k$ Šolt{\'e}s vertices for $k = 1, 2, \ldots, 8$.
A blue-coloured number means that we have provided drawings of these graphs (see Figures~\ref{fig:smallest_cubic_third} and~\ref{fig:small_examples}).}
\label{tbl:1}
$
\begin{array}{r|r||r|r|r|r|r|r|r|r}
 &  & \multicolumn{8}{c}{ N(\mathcal{R}^3_n, k)} \\
\cline{3-10} 
n & |\mathcal{R}^3_n| & k=1 &  k =2 & k = 3 & k = 4 & k = 5 &k = 6 &k = 7 & k =8 \\
\hline \hline
\leq 12 & 112 & \text{-} & \text{-} & \text{-} & \text{-} & \text{-} & \text{-} & \text{-} & \text{-} \\
14 & 509 & 4 & {\color{blue}\mathbf{3}} & \text{-} & \text{-} & \text{-} & \text{-} & \text{-} & \text{-} \\
16 & 4060 & 108 & 37 & {\color{blue}\mathbf{1}} &{\color{blue}\mathbf{2}} & \text{-} & \text{-} & \text{-} & \text{-} \\
18 & 41301 & 1014 & 200 & \text{-} &  {\color{blue}\mathbf{1}} & \text{-} & \text{-} & \text{-} & \text{-} \\
20 & 510489 & 13460 & 1076 & 6 & 13 & \text{-} & \text{-} & \text{-} & \text{-} \\
22 & 7319447 & 194432 & 9610 & 151 & 52 & \text{-} & 1 & \text{-} & \text{-} \\
24 & 117940535 & 3161124 & 130087 & 2596 & 333 & 2 & 3 & \text{-} & {\color{blue}\mathbf{1}} \\
\end{array}
$
\end{table}


\begin{table}[!htb]
\centering
\caption{The numbers of (non-isomorphic) quartic graphs of orders up to $17$ with Šolt{\'e}s vertices. Naming conventions
used in Table~\ref{tbl:1} also apply here.}
\label{tbl:2}
$
\begin{array}{r|r||r|r|r|r}
 &  & \multicolumn{4}{c}{ N(\mathcal{R}^4_n, k)} \\
\cline{3-6} 
n & |\mathcal{R}^4_n| & k = 1 & k = 2 & k = 3 & k = 4 \\
\hline \hline
\leq 12 & 1894 & \text{-} & \text{-} & \text{-} & \text{-} \\
13 & 10778 & \text{-} & {\color{blue}\mathbf{1}} & \text{-} & \text{-} \\
14 & 88168 & 30 & 6 & \text{-} & \text{-} \\
15 & 805491 & 265 & 85 & \text{-} & 5 \\
16 & 8037418 & 2191 & 472 & \text{-} & \text{-} \\
17 & 86221634 & 14430 & 2097 & 4 & 1 \\
\end{array}
$
\end{table}

\begin{table}[!htb]
\centering
\caption{The numbers of (non-isomorphic) quintic graphs of order up to $14$ with Šolt{\'e}s vertices. Naming conventions
used in Table~\ref{tbl:1} also apply here.}
\label{tbl:3}
$
\begin{array}{r|r||r|r|r|r}
 &  & \multicolumn{4}{c}{ N(\mathcal{R}^5_n, k)} \\
\cline{3-6} 
n & |\mathcal{R}^5_n| & k = 1 & k = 2 & k = 3 & k = 4 \\
\hline \hline
\leq 12 & 7912 & \text{-} & \text{-} & \text{-} & \text{-} \\
14 & 3459383 & 8 & 3 & \text{-}  & \text{-}  \\
\end{array}
$
\end{table}

\begin{figure}[!htb]
\centering
\subcaptionbox{\label{subfig:a}}
{\begin{tikzpicture}[scale=1.2]
\tikzstyle{edge}=[draw,thick]
\tikzstyle{every node}=[draw, circle, fill=blue!50!white, inner sep=1.5pt]
\node (v7) at (-0.7,1.4) {};
\node (v8) at (0.7,1.4) {};
\node (v1) at (0,1) {};
\node (v9) at (0,0.5) {};
\node (v2) at (-1,0.5) {};
\node (v3) at (1,0.5) {};
\path[edge] (v2) -- (v7) -- (v8) -- (v3); 
\path[edge] (v7) -- (v1) -- (v8); 
\path[edge] (v1) -- (v9); 
\path[edge] (v2) -- (v9) -- (v3); 
\node[fill=red!70!white] (v10) at (-1.2,-0.3) {};
\node[fill=red!70!white] (v12) at (1.2,-0.3) {};
\node (v0) at (-0.6,-0.4) {};
\node (v6) at (-0.6,-1.1) {};
\node (v11) at (0,-0.4) {};
\node (v13) at (0,-1.1) {};
\node (v4) at (0.6,-0.4) {};
\node (v5) at (0.6,-1.1) {};
\path[edge] (v2) -- (v10) -- (v6) -- (v13) -- (v5) -- (v12) -- (v3); 
\path[edge] (v10) -- (v0) -- (v11) -- (v4) -- (v12); 
\path[edge] (v0) -- (v6); 
\path[edge] (v11) -- (v13); 
\path[edge] (v5) -- (v4); 
\end{tikzpicture}}
\quad %
\subcaptionbox{\label{subfig:b}}
{\begin{tikzpicture}[scale=1.2]
\tikzstyle{edge}=[draw,thick]
\tikzstyle{every node}=[draw, circle, fill=blue!50!white, inner sep=1.5pt]
\node (v7) at (-0.7,1.4) {};
\node (v8) at (0.7,1.4) {};
\node (v1) at (0,1) {};
\node (v9) at (0,0.5) {};
\node (v2) at (-1,0.5) {};
\node (v3) at (1,0.5) {};
\path[edge] (v2) -- (v7) -- (v8) -- (v3); 
\path[edge] (v7) -- (v1) -- (v8); 
\path[edge] (v1) -- (v9); 
\path[edge] (v2) -- (v9) -- (v3); 
\node[fill=red!70!white] (v10) at (-1.2,-0.3) {};
\node[fill=red!70!white] (v12) at (1.2,-0.3) {};
\node (v0) at (-0.6,-0.4) {};
\node (v6) at (-0.6,-1.1) {};
\node (v11) at (0,-0.4) {};
\node (v13) at (0,-1.1) {};
\node (v4) at (0.6,-0.4) {};
\node (v5) at (0.6,-1.1) {};
\path[edge] (v2) -- (v10) -- (v6) -- (v13) -- (v5) -- (v12) -- (v3); 
\path[edge] (v10) -- (v0) -- (v11) -- (v4) -- (v12); 
\path[edge] (v0) -- (v6); 
\path[edge] (v11) -- (v5); 
\path[edge] (v13) -- (v4); 
\end{tikzpicture}}
\quad %
\subcaptionbox{\label{subfig:c}}
{\begin{tikzpicture}[scale=1.2]
\tikzstyle{edge}=[draw,thick]
\tikzstyle{every node}=[draw, circle, fill=blue!50!white, inner sep=1.5pt]
\node (v6) at (0,0.6) {};
\node (v13) at (0,1) {};
\node (v8) at (-0.4,1.3) {};
\node (v3) at (0.4,1.3) {};
\node (v11) at (-0.9,0.8) {};
\node (v1) at (0.9,0.8) {};
\node (v2) at (-0.7,1.8) {};
\node (v12) at (0.7,1.8) {};
\path[edge] (v6) -- (v13) -- (v8) -- (v3) -- (v13); 
\path[edge] (v11) -- (v6) -- (v1);
\path[edge] (v2) -- (v8);
\path[edge] (v12) -- (v3); 
\node[fill=red!70!white] (v4) at (-0.7,-0.0) {};
\node[fill=red!70!white] (v10) at (0.7,-0.0) {};
\node (v9) at (-0.6,-0.8) {};
\node (v7) at (0.6,-0.8) {};
\node (v0) at (0,-0.5) {};
\node (v5) at (0,-1.1) {};
\path[edge] (v4) -- (v11) -- (v2) -- (v12) -- (v1) -- (v10); 
\path[edge] (v10) -- (v4); 
\path[edge] (v4) -- (v9) -- (v0) -- (v7) -- (v10);
\path[edge] (v9) -- (v5) -- (v7); 
\path[edge] (v0) -- (v5); 
\end{tikzpicture}}
\quad %
\subcaptionbox{\label{subfig:d}}
{\begin{tikzpicture}[scale=1.2]
\tikzstyle{edge}=[draw,thick]
\tikzstyle{every node}=[draw, circle, fill=blue!50!white, inner sep=1.5pt]
\node[fill=red!70!white] (v4) at (-1,-0.0) {};
\node[fill=red!70!white] (v10) at (1,-0.0) {};
\node[fill=red!70!white] (v8) at (-1,1) {};
\node[fill=red!70!white] (v11) at (1,1) {};
\node (v9) at (-0.6,-0.8) {};
\node (v7) at (0.6,-0.8) {};
\node (v0) at (0,-0.5) {};
\node (v5) at (0,-1.1) {};
\path[edge] (v9) -- (v5) -- (v7); 
\path[edge] (v0) -- (v5); 
\path[edge] (v4) -- (v9) -- (v0) -- (v7) -- (v10); 
\node (v3) at (-0.6,{0.8+1}) {};
\node (v6) at (0.6,{0.8+1}) {};
\node (v2) at (0,{0.5+1}) {};
\node (v1) at (0,{1.1+1}) {};
\path[edge] (v1) -- (v2); 
\path[edge] (v3) -- (v1) -- (v6) -- (v2) -- (v3); 
\node (v12) at (-0.3,0.8) {};
\node (v13) at (0.3,0.8) {};
\node (v14) at (-0.3,0.2) {};
\node (v15) at (0.3,0.2) {};
\path[edge] (v3) -- (v8) -- (v4); 
\path[edge] (v6) -- (v11) -- (v10); 
\path[edge] (v12) -- (v13) -- (v15) -- (v14) -- (v12); 
\path[edge] (v12) -- (v8);
\path[edge] (v14) -- (v4); 
\path[edge] (v13) -- (v11);
\path[edge] (v15) -- (v10); 
\end{tikzpicture}}
\quad %
\subcaptionbox{\label{subfig:e}}
{\begin{tikzpicture}[scale=1.2]
\tikzstyle{edge}=[draw,thick]
\tikzstyle{every node}=[draw, circle, fill=blue!50!white, inner sep=1.5pt]
\node[fill=red!70!white] (v1) at (-1,0.0) {};
\node[fill=red!70!white] (v2) at (1,0.0) {};
\node[fill=red!70!white] (v9) at (-1,-0.6) {};
\node[fill=red!70!white] (v10) at (1,-0.6) {};
\node (v4) at (-0.8,1.2) {};
\node (v3) at (0.8,1.2) {};
\node (v5) at (-0.4, 0.3) {};
\node (v6) at (0.4, 0.3) {};
\node (v8) at (-0.4,0.8) {};
\node (v7) at (0.4,0.8) {};
\node (v12) at (-0.8,-1.2-0.6) {};
\node (v11) at (0.8,-1.2-0.6) {};
\node (v14) at (-0.4, -0.3-0.6) {};
\node (v15) at (0.4, -0.3-0.6) {};
\node (v13) at (-0.4,-0.8-0.6) {};
\node (v16) at (0.4,-0.8-0.6) {};
\path[edge] (v1) -- (v9);
\path[edge] (v2) -- (v10);
\path[edge] (v1) -- (v4) -- (v3) -- (v2); 
\path[edge] (v9) -- (v12) -- (v11) -- (v10); 
\path[edge] (v5) -- (v6) -- (v7) -- (v8) -- (v5); 
\path[edge] (v13) -- (v14) -- (v15) -- (v16) -- (v13);
\path[edge] (v9) -- (v14);
\path[edge] (v10) -- (v15);
\path[edge] (v11) -- (v16);
\path[edge] (v12) -- (v13);
\path[edge] (v1) -- (v5);
\path[edge] (v4) -- (v8);
\path[edge] (v3) -- (v7);
\path[edge] (v2) -- (v6);
\end{tikzpicture}}
\subcaptionbox{\label{subfig:f}}
{\begin{tikzpicture}[scale=1.2]
\tikzstyle{edge}=[draw,thick]
\tikzstyle{every node}=[draw, circle, fill=blue!50!white, inner sep=1.5pt]
\node (v1) at (0,0) {};
\node (v2) at (30:0.4) {};
\node (v3) at (150:0.4) {};
\node (v4) at (-90:0.4) {};
\node (v5) at ($ (v4) + (-60:0.5) $) {};
\node (v6) at ($ (v4) + (-120:0.5) $) {};
\node (v7) at ($ (v5) + (-120:0.5) $) {};
\node[fill=red!70!white] (v8) at ($ (v7) + (-90:0.4) $) {};
\path[edge] (v5) -- (v6) -- (v4) -- (v5) -- (v7) -- (v6);
\node (v9) at ($ (v2) + (0:0.5) $) {};
\node (v10) at ($ (v2) + (60:0.5) $) {};
\node (v11) at ($ (v9) + (60:0.5) $) {};
\path[edge] (v9) -- (v10) -- (v2) -- (v9) -- (v11) -- (v10);
\node[fill=red!70!white] (v12) at ($ (v11) + (30:0.4) $) {};
\node (v13) at ($ (v3) + (120:0.5) $) {};
\node (v14) at ($ (v3) + (180:0.5) $) {};
\node (v15) at ($ (v14) + (120:0.5) $) {};
\node[fill=red!70!white] (v16) at ($ (v15) + (150:0.4) $) {};
\path[edge] (v13) -- (v14) -- (v3) -- (v13) -- (v15) -- (v14);
\path[edge] (v1) -- (v2);
\path[edge] (v1) -- (v3);
\path[edge] (v1) -- (v4);
\path[edge] (v7) -- (v8);
\path[edge] (v11) -- (v12);
\path[edge] (v15) -- (v16);
\path[edge] (v8) to[bend right=30] (v12);
\path[edge] (v12) to[bend right=30] (v16);
\path[edge] (v16) to[bend right=30] (v8);
\end{tikzpicture}}
\quad %
\subcaptionbox{\label{subfig:g}}
{\begin{tikzpicture}[scale=1.3]
\tikzstyle{edge}=[draw,thick]
\tikzstyle{every node}=[draw, circle, fill=blue!50!white, inner sep=1.5pt]
\node[fill=red!70!white] (v9) at (0,0) {};
\node (v11) at (0,-0.6) {};
\node (v13) at (0,-1.2) {};
\node[fill=red!70!white] (v15) at (0,-1.8) {};
\node[fill=red!70!white] (v10) at (2,0) {};
\node (v12) at (2,-0.6) {};
\node (v14) at (2,-1.2) {};
\node[fill=red!70!white] (v16) at (2,-1.8) {};
\node (v17) at (1,-0.6) {};
\node (v18) at (1,-1.2) {};
\path[edge] (v9) -- (v10) -- (v11) -- (v17) -- (v12) -- (v9);
\path[edge] (v13) -- (v18) -- (v14) -- (v16) -- (v15) -- (v13);
\path[edge] (v11) -- (v13);
\path[edge] (v17) -- (v18);
\path[edge] (v12) -- (v14);
\node (v1) at (-0.8,-0.4) {};
\node (v2) at (-0.8,-1.4) {};
\node (v3) at (-0.8-0.2,-0.9) {};
\node (v4) at (-0.8+0.2,-0.9) {};
\path[edge] (v3) -- (v4) -- (v2) -- (v3) -- (v1) -- (v4);
\node (v5) at (2+0.8,-0.4) {};
\node (v6) at (2+0.8,-1.4) {};
\node (v7) at (2+0.8-0.2,-0.9) {};
\node (v8) at (2+0.8+0.2,-0.9) {};
\path[edge] (v7) -- (v8) -- (v5) -- (v7) -- (v6) -- (v8);
\path[edge] (v1) -- (v9);
\path[edge] (v2) -- (v15);
\path[edge] (v5) -- (v10);
\path[edge] (v6) -- (v16);
\end{tikzpicture}}
\quad %
\subcaptionbox{\label{subfig:h}}
{\begin{tikzpicture}[scale=1.6]
\tikzstyle{edge}=[draw,thick]
\tikzstyle{every node}=[draw, circle, fill=blue!50!white, inner sep=1.5pt]
\node (v1) at (-90:0.5) {};
\node (v2) at (-90+72:0.5) {};
\node (v3) at (-90+72*2:0.5) {};
\node (v4) at (-90+72*3:0.5) {};
\node (v5) at (-90+72*4:0.5) {};
\path[edge] (v2) -- (v3) -- (v4) -- (v5);
\path[edge] (v1) -- (v3) -- (v5) -- (v2) -- (v4) -- (v1);
\node[fill=red!70!white] (v6) at (-0.8,-0.8) {};
\node[fill=red!70!white] (v7) at (0.8,-0.8) {};
\node (v10) at (0,-1.3) {};
\node (v11) at (0,-1.6) {};
\node (v12) at (0.4,-1.2) {};
\node (v13) at (0.5,-1.5) {};
\node (v8) at (-0.4,-1.2) {};
\node (v9) at (-0.5,-1.5) {};
\path[edge] (v6) -- (v5);
\path[edge] (v7) -- (v2);
\path[edge] (v6) -- (v1) -- (v7);
\path[edge] (v6) -- (v8) -- (v10) -- (v12) -- (v7);
\path[edge] (v6) -- (v9) -- (v11) -- (v13) -- (v7);
\path[edge] (v8) -- (v9);
\path[edge] (v12) -- (v13);
\path[edge] (v8) -- (v11) -- (v12);
\path[edge] (v9) -- (v10) -- (v13);
\end{tikzpicture}}
\caption{Examples of small regular graphs with two or more Šolt{\'e}s vertices: (a) to (c) are the three cubic graphs of order $14$ with two Šolt{\'e}s vertices; (d) and (e) are the two cubic graphs of order $16$ with four Šolt{\'e}s vertices;
(f) is the only cubic graph of order $16$ with three Šolt{\'e}s vertices; (g) is the only cubic graph of order $18$ with four Šolt{\'e}s  vertices; (h) is the only quartic graph of order $13$ with two Šolt{\'e}s vertices.
Šolt{\'e}s vertices are coloured red.}
\label{fig:small_examples}
\end{figure}
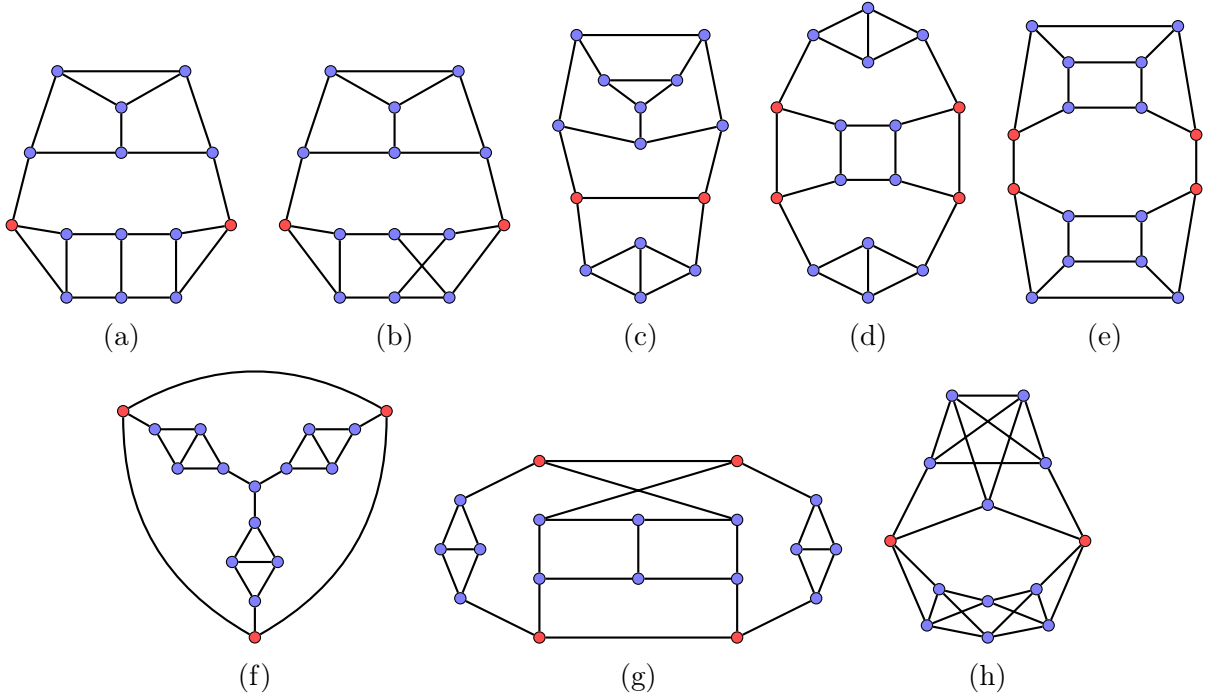

In the next two sections we construct an infinite family of cubic $2$-connected graphs with at least $2^r$, $r \geq 1$, Šolt{\'e}s vertices.
It recently came to our attention that Dobrynin independently found an infinite family of cubic graphs with four 
Šolt{\'e}s vertices~\cite{Dobrynin2023}. However, our method is more general.

\bigskip

\clearpage
\section{Cubic \texorpdfstring{$\boldsymbol{2}$}{2}-connected graphs with two Šolt{\'e}s vertices}

In the present section we prove the following result.

\begin{theorem}
\label{thm:2}
There exist infinitely many cubic $2$-connected graphs $G$ which contain at least two Šolt{\'e}s vertices.
\end{theorem}

We prove Theorem~{\ref{thm:2}} by a sequence of lemmas.
We start by giving several definitions.
First, we define a graph $G_t$ on $8t+8$ vertices, where $t\ge 1$.
Take $2t$ copies of the diamond graph (i.e.\ $K_4 - e$) and connect their degree-$2$ vertices, so that a ring of $2t$ copies of $K_4 - e$ is formed.
Add a disjoint $4$-cycle to that graph.
Then subdivide one of the edges that connects two consecutive diamonds by two vertices, denote them by $z_1$ and $z_2$, and connect $z_1$ and $z_2$ with two opposite vertices of the $4$-cycle.
Add a leaf to each remaning degree-$2$ vertex of the $4$-cycle.
Denote the resulting graph by $G_t$.
For an illustration, see $G_3$ in Figure~\ref{fig:gtgraph}.
Note that $G_t$ has exactly two leaves, denote them by $v_1$ and $v_2$, and all the remaining vertices have degree~$3$.
Denote by $u_1$ and $u_2$ the two vertices at the longest distance from $v_1$.
This distance is $d_G(v_1, u_1) = d_G(v_1, u_2) = 3t+3$ and also $d_G(v_2, u_1) = d_G(v_2, u_2) = 3t+3$.
Observe that $G_t$ has an automorphism (a symmetry) fixing both $v_1$ and $v_2$, while interchanging $u_1$ with~$u_2$.

\begin{figure}[!htb]
\centering
\begin{tikzpicture}[scale=1.2]
\tikzstyle{edge}=[draw,thick]
\tikzstyle{every node}=[draw, circle, fill=blue!50!white, inner sep=1.5pt]
\foreach \i in {1,...,6} { 
	\pgfmathsetmacro{\kot}{360 * \i / 7}
	\coordinate (c\i) at (\kot:1.5);
	\node (a\i_1) at ($ (c\i) + (\kot:0.25) $) {};
	\node (a\i_2) at ($ (c\i) + (\kot:-0.25) $) {};
	\node (a\i_3) at ($ (c\i) + ({\kot + 90}:0.4) $) {};
	\node (a\i_4) at ($ (c\i) + ({\kot + 90}:-0.4) $) {};
	\path[edge] (a\i_1) -- (a\i_2) -- (a\i_3) -- (a\i_1) -- (a\i_4) -- (a\i_2);
}
\node[draw=none,fill=none] at (-1.8, 0.3) {$u_1$};
\node[draw=none,fill=none] at (-1.8, -0.3) {$u_2$};
\node[draw=none,fill=none] at (1.65, 0.55) {$z_1$};
\node[draw=none,fill=none] at (1.65, -0.6) {$z_2$};
\path[edge] (a1_3) -- (a2_4);
\path[edge] (a2_3) -- (a3_4);
\path[edge] (a3_3) -- (a4_4);
\path[edge] (a4_3) -- (a5_4);
\path[edge] (a5_3) -- (a6_4);
\node (b1) at ($ (0:1.5) + (90:0.35) $) {};
\node (b2) at ($ (0:1.5) + (90:-0.35) $) {};
\path[edge] (a1_4) -- (b1) -- (b2) -- (a6_3);
\node (c1) at ($ (b1) + (0.5, 0) $) {};
\node (c1a) at ($ (c1) + (0.5, 0) $) {};
\node[label=0:$v_1$]  (c1b) at ($ (c1a) + (0.5, 0) $) {};
\node (c2) at ($ (b2) + (0.5, 0) $) {};
\node (c2a) at ($ (c2) + (0.5, 0) $) {};
\node[label=0:$v_2$] (c2b) at ($ (c2a) + (0.5, 0) $) {};
\path[edge] (b1) -- (c1) -- (c1a) -- (c1b);
\path[edge] (b2) -- (c2) -- (c2a) -- (c2b);
\path[edge] (c1) -- (c2a);
\path[edge] (c2) -- (c1a);
\end{tikzpicture}
\caption{The graph $G_3$.}
\label{fig:gtgraph}
\end{figure}
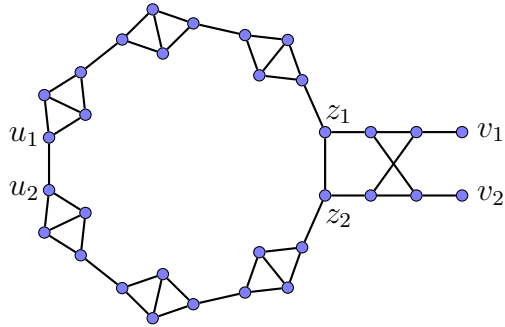

Now, we determine $f(t)=W(G_t-u_1)-W(G_t)$.
We compute $f(t)$ by summing the contributions of all vertices (first the contribution of quadruples of vertices of all copies of $K_4-e$, and then the contribution of the vertices which are not in any copy of $K_4-e$).
As the calculation is long and tedious, we present just the result
\begin{equation}
f(t)=16 t^3 - 8 t^2 - 26 t -14,
\end{equation}
which was checked by a computer.
Actually, the exact value of $f(t)$ is not important here.
The crucial property is that for $t$ big enough, $f(t)$ is positive.
In fact, $\lim_{t\to\infty}f(t)=\infty$; see also Section~\ref{sec:general},
where a lower bound for $f(t)$ is given.

To motivate the above definition, we briefly describe the main idea of the proof.
We attach trees $T_1$ and $T_2$ to vertices $v_1$ and $v_2$ of $G_t$, and then we add edges to them so that the resulting graph 
$H$ will be cubic and $2$-connected.
See Figure~\ref{fig:hgraph} for an example.
This will be done in three phases.
\begin{enumerate}[label=P\arabic*)]
\item
In the first phase, we will construct trees $T_1$ and $T_2$ to guarantee that vertices $u_1$ and $u_2$ are indeed Šolt{\'e}s vertices.
The vertices of $T_1$ and $T_2$ can be partitioned into several layers based on their distance to $v_1$ and $v_2$, respectively.
The resulting graph will be denoted by $Q$.
\item
In the second phase, we will add the `red' edges, whose endpoints are in two consecutive layers, to graph $Q$.
The resulting graph will be denoted by $R$.
The purpose of the second phase is to make sure that the sum of free valencies is even within each layer, making the next phase possible.
Note that although the final graph $H$ is cubic, the intermediate graphs, $Q$ and $R$, are subcubic. The \emph{free valency} of a vertex $v$ in a subcubic 
graph $G$ is $3 - \deg_G(v)$.
\item
In the last phase, we will add blue edges to $R$ in order to to obtain cycles and paths, so that the resulting graph $H$ will be cubic and $2$-connected.
The endpoints of `blue' edges will reside in the same layer of the forest $T_1 \cup T_2$. Adding `red' and `blue' edges has no influence on Šolt{\'e}sness of $u_1$ and $u_2$.
\end{enumerate}

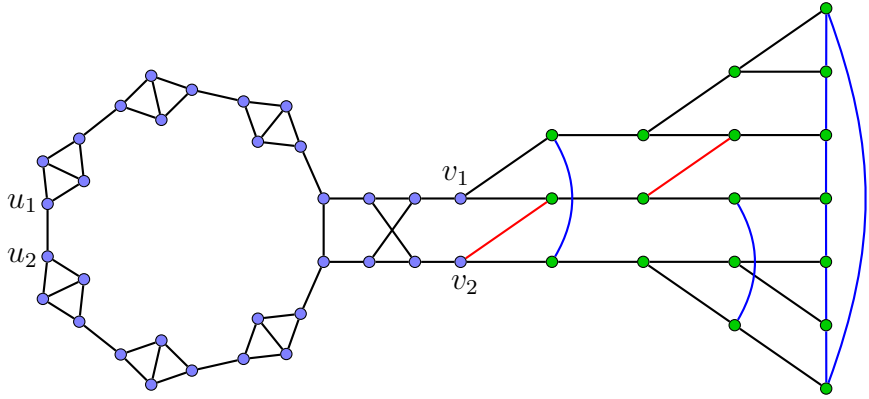
\begin{figure}[!htb]
\centering
\begin{tikzpicture}[scale=1.2]
\tikzstyle{edge}=[draw,thick]
\tikzstyle{treeVtx}=[fill=green!80!black]
\tikzstyle{every node}=[draw, circle, fill=blue!50!white, inner sep=1.5pt]
\foreach \i in {1,...,6} { 
	\pgfmathsetmacro{\kot}{360 * \i / 7}
	\coordinate (c\i) at (\kot:1.5);
	\node (a\i_1) at ($ (c\i) + (\kot:0.25) $) {};
	\node (a\i_2) at ($ (c\i) + (\kot:-0.25) $) {};
	\node (a\i_3) at ($ (c\i) + ({\kot + 90}:0.4) $) {};
	\node (a\i_4) at ($ (c\i) + ({\kot + 90}:-0.4) $) {};
	\path[edge] (a\i_1) -- (a\i_2) -- (a\i_3) -- (a\i_1) -- (a\i_4) -- (a\i_2);
}
\node[draw=none,fill=none] at (-1.8, 0.3) {$u_1$};
\node[draw=none,fill=none] at (-1.8, -0.3) {$u_2$};
\node[draw=none,fill=none] at (2.95, 0.6) {$v_1$};
\node[draw=none,fill=none] at (3.05, -0.6) {$v_2$};
\path[edge] (a1_3) -- (a2_4);
\path[edge] (a2_3) -- (a3_4);
\path[edge] (a3_3) -- (a4_4);
\path[edge] (a4_3) -- (a5_4);
\path[edge] (a5_3) -- (a6_4);
\node (b1) at ($ (0:1.5) + (90:0.35) $) {};
\node (b2) at ($ (0:1.5) + (90:-0.35) $) {};
\path[edge] (a1_4) -- (b1) -- (b2) -- (a6_3);
\node (c1) at ($ (b1) + (0.5, 0) $) {};
\node (c1a) at ($ (c1) + (0.5, 0) $) {};
\node  (c1b) at ($ (c1a) + (0.5, 0) $) {};
\node (c2) at ($ (b2) + (0.5, 0) $) {};
\node (c2a) at ($ (c2) + (0.5, 0) $) {};
\node (c2b) at ($ (c2a) + (0.5, 0) $) {};
\path[edge] (b1) -- (c1) -- (c1a) -- (c1b);
\path[edge] (b2) -- (c2) -- (c2a) -- (c2b);
\path[edge] (c1) -- (c2a);
\path[edge] (c2) -- (c1a);
\node[treeVtx] (l1x2) at ($ (c1b) + (1, 0) $) {};
\node[treeVtx] (l1x1) at ($ (l1x2) + (0, 0.7) $) {};
\node[treeVtx] (l1x3) at ($ (l1x2) + (0, -0.7) $) {};
\node[treeVtx] (l2x1) at ($ (l1x1) + (1, 0) $) {};
\node[treeVtx] (l2x2) at ($ (l1x2) + (1, 0) $) {};
\node[treeVtx] (l2x3) at ($ (l1x3) + (1, 0) $) {};
\node[treeVtx] (l3x2) at ($ (l2x1) + (1, 0) $) {};
\node[treeVtx] (l3x3) at ($ (l2x2) + (1, 0) $) {};
\node[treeVtx] (l3x4) at ($ (l2x3) + (1, 0) $) {};
\node[treeVtx] (l3x1) at ($ (l3x2) + (0, 0.7) $) {};
\node[treeVtx] (l3x5) at ($ (l3x4) + (0, -0.7) $) {};
\node[treeVtx] (l4x2) at ($ (l3x1) + (1, 0) $) {};
\node[treeVtx] (l4x3) at ($ (l3x2) + (1, 0) $) {};
\node[treeVtx] (l4x4) at ($ (l3x3) + (1, 0) $) {};
\node[treeVtx] (l4x5) at ($ (l3x4) + (1, 0) $) {};
\node[treeVtx] (l4x6) at ($ (l3x5) + (1, 0) $) {};
\node[treeVtx] (l4x7) at ($ (l4x6) + (0, -0.7) $) {};
\node[treeVtx] (l4x1) at ($ (l4x2) + (0, 0.7) $) {};
\path[edge] (c1b) -- (l1x1);
\path[edge] (c1b) -- (l1x2);
\path[edge] (c2b) -- (l1x3);
\path[edge] (l1x1) -- (l2x1);
\path[edge] (l1x2) -- (l2x2);
\path[edge] (l1x3) -- (l2x3);
\path[edge] (l2x1) -- (l3x1);
\path[edge] (l2x1) -- (l3x2);
\path[edge] (l2x2) -- (l3x3);
\path[edge] (l2x3) -- (l3x4);
\path[edge] (l2x3) -- (l3x5);
\path[edge] (l3x1) -- (l4x1);
\path[edge] (l3x1) -- (l4x2);
\path[edge] (l3x2) -- (l4x3);
\path[edge] (l3x3) -- (l4x4);
\path[edge] (l3x4) -- (l4x5);
\path[edge] (l3x4) -- (l4x6);
\path[edge] (l3x5) -- (l4x7);
\path[edge,color=red] (c2b) -- (l1x2);
\path[edge,color=red] (l2x2) -- (l3x2);
\path[edge,color=blue] (l1x1) edge[bend left=30] (l1x3);
\path[edge,color=blue] (l3x3) edge[bend left=30] (l3x5);
\path[edge,color=blue] (l4x1) -- (l4x2) -- (l4x3) -- (l4x4) -- (l4x5) -- (l4x6) -- (l4x7);
\path[edge,color=blue] (l4x1) edge[bend left=20] (l4x7);
\end{tikzpicture}
\caption{A graph $H$ with two Šolt{\'e}s vertices, namely $u_1$ and $u_2$, that contains $G_3$.}
\label{fig:hgraph}
\end{figure}

Let us consider the resulting graph $H$.
Observe that if $x \in V(H) \setminus V(G_t)$, then we have $w_{H}(x)-w_{H-u_1}(x)=d_{H}(u_1,x)$. Clearly, every $(u_1,x)$-path contains
one of the vertices from $\{v_1, v_2\}$.
Hence, for calculating $W(H)-W(H-u_1)$, only the distance of the new vertices $x$ to $\{v_1,v_2\}$ matters.

We need to find suitable trees $T_1$ and $T_2$ rooted at $v_1$ and $v_2$, respectively.
Each of these trees will have $q$ vertices and their depth, say $d$, will be determined later.
What next properties do $T_1$ and $T_2$ need to have?
Let $\ell_i$ be the number of vertices of the forest $T_1 \cup T_2$ at distance $i$, $1\le i\le d$, from $\{v_1,v_2\}$.
Since the resulting graph will be cubic and $2$-connected, we have $2\le \ell_1\le 4, 2\le \ell_2\le 8, 2\le \ell_3\le 2^4,\dots$ 
and for the last value $\ell_d$ we have $1\le\ell_d\le 2^{d+1}$.
The trees attached to $v_1$ and $v_2$ may be paths in which case we get $\ell_1 = \ell_2 = \cdots = \ell_q = 2$ (since each of $T_1$ and $T_2$ have exactly $q$ vertices).
In this case the transmission of $v_j$ in $T_j$ is biggest possible, $1\le j\le 2$.
In the other extremal situation the transmission of $v_j$ in $T_j$ is smallest possible, 
$1\le j\le 2$, which results in $\ell_i = 2^{i+1}$, $1\le i<d$.
If $x\in V(H)\setminus V(G_t)$ is at distance $i$ from $\{v_1,v_2\}$, then $d_{H}(u_1,x)=3t+3+i$.
Denote by $D$ the sum of distances from all vertices of $V(H)\setminus V(G_t)$ to $u_1$.
Then
\begin{equation}
D=\sum_{i=1}^d(3t+3+i)\ell_i.
\end{equation}
Observe that we need to find a finite sequence $(\ell_1,\ell_2, \ldots,\ell_d)$ so that $f(t)=D$.
As the resulting graph $H$ has to be cubic, we need to add an even number of vertices, $2q$, to the graph $G_t$.
What are the bounds for $D$?

First, we determine the lower bound; let us denote it by $D_m$.
This bound will be obtained when $\ell_i$ attains the maximum possible value of $2^{i + 1}$ for every $1 \leq i \leq d - 1$.
In other words, we are attaching complete binary trees to $v_1$ and $v_2$.
Let $a=\lfloor\log_2(2q+3)\rfloor$. Recall that the depth of a complete binary tree with $n$ vertices is $\lfloor \log_2(n) \rfloor$ and note that in our case $a - 1= d$. Then
\begin{align*}
\ell_1&=4,\\
\ell_2&=8,\\
& \;\; \vdots\\
\ell_{a-2}&=2^{a-1},\\
\ell_{a-1}&=2q-\sum_{i=1}^{a-2} \ell_i=2q-2^a+4.
\end{align*}
The above sequence will be called the \emph{short sequence} and denoted by $L_m$. 
Using the formula $\sum_{i=1}^aix^{i-1}=(ax^{a+1}-(a+1)x^a+1)/(x-1)^2$, we get
\begin{equation}
\begin{aligned}
D_m &=\sum_{i=1}^{a-2}(3t+3+i)2^{i+1}+(3t+a+2)(2q-2^a+4)\\
&=(3t+1)2q+\sum_{i=1}^{a-2} (i+2)2^{i+1}+(a+1)(2q-2^a+4)+2\cdot 2^1+1\cdot 2^0-5\\
&=(3t+1)2q+\sum_{i=1}^a i 2^{i-1}+(a+1)(2q-2^a+4)-5\\
&=(3t+1)2q+a2^{a+1}-(a+1)2^a+1+(a+1)2q-(a+1)2^a+4(a+1)-5\\
&=(3t+1)2q+a2^{a+1}-(a+1)2^{a+1}+(a+1)2q+4a\\
&=(3t+1)2q-2^{a+1}+(a+1)2q+4a.
\end{aligned}
\end{equation}

Now we find the upper bound $D^m$ for $D$.
In this case $\ell_1=\ell_2=\dots=\ell_q=2$; this sequence will be called the \emph{long sequence} and denoted by $L^m$. Therefore,
\begin{equation}
\begin{aligned}
D^m&=2(3t+4)+2(3t+5)+\dots+2(3t+3+q)\\
&=(3t+3)2q+2\sum_{i=1}^qi\\
&=(3t+1)2q+q^2+5q.
\end{aligned}
\end{equation}
Note that $D_m$ and $D^m$ are functions of $q$ and $t$.

\begin{lemma}
\label{lem:qranges}
Let $t \geq 3$. Then there exists $q$, such that $D_m \leq f(t) \leq D^m$.
\end{lemma}

\begin{proof}
Observe that if $q\sim 4t\sqrt t$ then  we have $D_m \leq f(t) \leq D^m$ for large enough $t$.
For small values of $t$ we computed the minimum and maximum value of $q$ that satisfies the condition
of the lemma; see Table~\ref{tbl:qranges}.
\end{proof}

Note that the value $q$ in Lemma~\ref{lem:qranges} is uniquely determined only for $t = 3$.
For larger values of $t$ we get a range of options. Any $q$ between $q_\text{min}$ and $q_\text{max}$ can be
used. Moreover, different values of $q$ lead to non-isomorphic graphs $H$.

\begin{table}
$
\begin{array}{r||rrrrrrrrrrrrrrrr}
t & 3 & 4 & 5 & 6 & 7 & 8 & 9 & 10 & 11 & 12 & 13 & 14 & 15 & 16 & 17 & 18  \\ 
\hline
q_\text{min} & 9 & 17 & 27 & 38 & 50 & 64 & 78 & 94 & 110 & 128 & 146 & 165 & 185 & 205 & 227 & 249 \\ 
\hline
q_\text{max} & 9 & 21 & 38 & 59 & 85 & 116 & 152 & 192 & 238 & 288 & 344 & 405 & 471 & 542 & 617 & 698 
\end{array}
$

\vspace{\baselineskip}
$
\begin{array}{r||rrrrrrrrrrrrr}
t & 19 & 20 & 21 & 22 & 23 & 24 & 25 & 26 & 27 & 28 & 29 & 30 & 31 \\
\hline
q_\text{min} &  272 & 295 & 319 & 344 & 370 & 396 & 422 & 450 & 478 & 506 & 535 & 565 & 595  \\ 
\hline
q_\text{max} & 785 & 876 & 973 & 1075 & 1182 & 1294 & 1411 & 1533 & 1661 & 1795 & 1933 & 2077 & 2225
\end{array}
$
\caption{The minimum and maximum value of $q$ which satisfy the condition of Lemma~\ref{lem:qranges}.}
\label{tbl:qranges}
\end{table}

Now, we show that for every integer $D$, $D_m\le D\le D^m$, there exists a finite sequence $(\ell_1,\ell_2, \ldots, \ell_d)$ which realises $D$.
Moreover, the graph $G_t$ can be extended to $H$ by attaching trees $T_1$ and $T_2$ that have 
$\ell_i$ vertices at distance $i$ from $\{v_1,v_2\}$, such that $H$ is a $2$-connected cubic graph.

Let us define an operation $\mathcal{M}$ that modifies one such sequence $L = (\ell_1,\ell_2, \ldots, \ell_d)$.

\begin{definition}
\label{def:operation}
Let $L = (\ell_1,\ell_2, \ldots, \ell_d)$.
Set $\ell_{d+1}=0$.
Let $i$ be the smallest value such that $\ell_i\ge 3$ and either 
\begin{enumerate}[label=(\roman*)]
\item
$2(\ell_i-1)>\ell_{i+1}+1$ or 
\item
$\ell_i=\ell_{i+1}=3$.
\end{enumerate}
We say that $\mathcal{M}(L) = (\ell'_1,\ell'_2, \ldots, \ell'_d)$ is a {\em modification} of sequence $L$ if $\ell_j'=\ell_j$ for all $j$, $1\le j\le d+1$, except for $j\in\{i,i+1\}$, for which $\ell_i'=\ell_i-1$ and $\ell_{i+1}'=\ell_{i+1}+1$.
\end{definition}

\begin{lemma}
\label{lem:sequenceGiver}
Let $t \geq 1$ and $q \geq 0$. For every $D$, $D_m\le D\le D^m$, there exists 
a finite sequence $L = (\ell_1, \ell_2, \ldots, \ell_d)$, such that
\begin{enumerate}[label=(\roman*)]
\item $\sum_{i=1}^{d} \ell_i = 2q$; 
\item $2 \leq \ell_i \leq 2^{i+1}$ for $i < d$ and $1 \leq \ell_d \leq 2^{d+1}$; 
\item $\ell_{i+1} \leq 2\ell_i$.
\end{enumerate}
Namely, $L = \mathcal{M}^{D - D_m}(L_m)$.
\end{lemma}

\begin{proof}
We start with the short sequence $L_m=(4,8,16,\dots)$. It clearly satisfies conditions (i) to (iii). We have already
seen that $L_m$ realises $D_m$. This established the base of induction.

Every other $D$ can be realised by a sequence that is obtained from $L_m$ by iteratively applying operation $\mathcal{M}$.
Assume that after $(D-1)-D_m$ steps we obtained the sequence $L=(\ell_1,\ell_2,\dots,\ell_d)$ which realises $D-1$.
Obviously, $\mathcal{M}(L)$ realises $D$. It is easy to check that conditions (i) to (iii) are satisfied for $\mathcal{M}(L)$.
\end{proof}

Next, we prove two additional properties that hold when operation $\mathcal{M}$ is iteratively applied on $L_m$.

\begin{lemma}
\label{lem:technicalLemma}
Let $t \geq 1$ and $q \geq 0$. Let $L = (\ell_1,\ell_2, \ldots, \ell_d)$ be obtained from $L_m$ by iteratively applying operation $\mathcal{M}$.
Then the following holds:
\begin{enumerate}[label=(\roman*)]
\item If (ii) of Definition~\ref{def:operation} applies, then $\ell_j = 2$ for all $j < i$.
\item If index $i$ in Definition~\ref{def:operation} is such that $\ell_{i+1} = 0$, then $\ell_{i} \geq 4$.
\end{enumerate}
\end{lemma}

\begin{proof}
(i): Observe that if part (ii) of Definition~\ref{def:operation} applies and $i\ge 2$, then $\ell_{i-1}=2$, since otherwise $i-1$ satisfies 
the assumption of the definition, thus $i$ is not the smallest such value.
Moreover, if there is $k$, $k<i$, with $\ell_k\ge 3$, then choose the largest possible $k$ with this property.
Then $\ell_{k+1}=\ell_{k+2}=\dots=\ell_{i-1}=2$.
Hence, $k$ satisfies the assumption of the definition, a contradiction.
It means that if part (ii) of the definition applies, then $\ell_{i-1}=\ell_{i-2}=\dots=\ell_1=2$.

(ii): If $\ell_{i+1}=0$ then clearly $\ell_i\ge 3$.
Suppose that $\ell_i=3$. Then $\ell_{i-1}=2$, since otherwise the assumptions apply to $i-1$.
As $2q$ is even, there must be $k<i$ such that $\ell_k\ge 3$.
Let $k$ be the largest possible value with this property.
Then the assumptions apply to $k$, a contradiction.
Thus, if $\ell_{i+1}=0$ then $\ell_{i}\ge 4$. 
\end{proof}

Note that part (ii) of Lemma~\ref{lem:technicalLemma} means that in the sequence $L' = \mathcal{M}(L)$,
$\ell_{i}'\ge 3$ and $\ell_{i+1}' = 1$. The corresponding graph $H$ will thus have a single vertex of degree $3$ in the final layer.

\begin{example}
For an illustration, the sequence of sequences
$$L_m,\ \mathcal{M}(L_m),\  \mathcal{M}^2(L_m),\  \mathcal{M}^3(L_m),\  \ldots$$
for $2q=20$ is 
\begin{align*}
& (4, 8, 8), & & (4, 5, 6, 5), & & (3, 4, 5, 6, 2), & & (2, 3, 4, 7, 4), & & (2, 2, 3, 5, 7, 1), \\
& (4, 7, 9), & & (4, 4, 7, 5), & & (3, 4, 4, 7, 2), & & (2, 3, 4, 6, 5), & & (2, 2, 3, 5, 6, 2), \\
& (4, 6, 10), & & (3, 5, 7, 5), & & (3, 3, 5, 7, 2), & & (2, 3, 4, 5, 6), & & (2, 2, 3, 4, 7, 2), \\
& (4, 6, 9, 1), & & (3, 5, 6, 6), & & (2, 4, 5, 7, 2), & & (2, 3, 4, 4, 7), & & (2, 2, 3, 4, 6, 3), \\
& (4, 6, 8, 2), & & (3, 4, 7, 6), & & (2, 4, 5, 6, 3), & & (2, 3, 3, 5, 7), & & (2, 2, 3, 4, 5, 4), \\
& (4, 5, 9, 2), & & (3, 4, 6, 7), & & (2, 4, 4, 7, 3), & & (2, 2, 4, 5, 7), & & \text{etc.} \\
& (4, 5, 8, 3), & & (3, 4, 5, 8), & & (2, 3, 5, 7, 3), & & (2, 2, 4, 5, 6, 1), & & \\
& (4, 5, 7, 4), & & (3, 4, 5, 7, 1), & & (2, 3, 5, 6, 4), & & (2, 2, 4, 4, 7, 1), & &  
\end{align*}
There are altogether $67$ sequences since $D^m-D_m=66$ when $2q=20$.
\end{example}

\begin{lemma}
Let $t \geq 3$. There exist rooted trees $T_1$ and $T_2$ such that vertices $u_1$ and $u_2$ are Šolt{\'e}s vertices in
the graph $Q$ obtained from $G_t$ by attaching $T_1$ and $T_2$ to vertices $v_1$ and~$v_2$. 
\end{lemma}

\begin{proof}
By Lemma~\ref{lem:qranges}, there exists $q$ such that $D_m \leq f(t) \leq D^m$. 
From Lemma~\ref{lem:sequenceGiver}, we obtain the sequence $L$, which gives us the appropriate number of vertices in 
every layer of the forest $T = T_1 \cup T_2$. This ensures that $v_1$ and $v_2$ are Šolt{\'e}s vertices in $Q$.

Now, we construct a graph $Q$ containing $G_t$ and realising $L$.
Let $T_j$ be the tree rooted at $v_j$, $1\le j\le 2$.
Then $T_1$ will have $\lceil\ell_i/2\rceil$ vertices at distance $i$ from $v_1$ and $T_2$ will have $\lfloor\ell_i/2\rfloor$ vertices at distance $i$ from $v_2$.
Observe that it is possible to construct both $T_1$ and $T_2$.
We have two possibilities.

\vspace{0.5\baselineskip}
\noindent
{\bf Case 1:} \emph{$\ell_i=2\ell_{i-1}$.}
Then either $\ell_i=2^{i+1}, \ell_{i-1}=2^i,\ldots,\ell_1=4$ and on the first $i$ levels both $T_1$ and $T_2$ are complete binary trees of height $i$; or $\ell_1=\ell_2=\dots=\ell_{i-1}=2$ and $\ell_i=4$, which means that both $T_1$ and $T_2$ contain one vertex at levels $1,2,\dots,i-1$ and two vertices at level $i$.

\vspace{0.5\baselineskip}
\noindent
{\bf Case 2:} \emph{$\ell_i\le 2\ell_{i-1}-1$.}
If $\ell_{i-1}$ is even then we can construct $i$-th level of both $T_1$ and $T_2$, and at least one vertex of level $i-1$ of $T_2$ will have degree less than $3$.
On the other hand, if $\ell_{i-1}$ is odd then $T_2$ has only $(\ell_{i-1}-1)/2$ vertices at level $i-1$.
However, it has $\lfloor\ell_i/2\rfloor$ vertices at level $i$  and $\lfloor\ell_i/2\rfloor\le \lfloor(2\ell_{i-1}-1)/2\rfloor=\ell_{i-1}-1=2\lfloor\ell_{i-1}/2\rfloor$, so in $T_2$, the number of vertices at level $i$ is at most twice the number of vertices at level $i-1$.
In this case, at least one vertex at level $i-1$ in $T_1$ will have degree less than $3$.
This concludes Case 2.
\end{proof}

We construct the trees $T_1$ and $T_2$ so that at each level we minimise the number of vertices of degree $3$.
Observe that then there is no level in which there are vertices of degree $1$ and also vertices of degree $3$.

Consider $v_1$ and $v_2$ as vertices of level $0$, and set $\ell_0=2$.
We plan to add edges within levels (i.e.\ `blue' edges) to create a cubic graph, but sometimes we must also add edges between consecutive levels (i.e.\ `red' edges).
First, we add necessary edges connecting vertices of different levels.
For every $i\ge 1$, if $\sum_{j=0}^i \ell_j$ is odd then add an edge joining a vertex (of degree $\le 2$) of $(i-1)$-th level with a vertex of $i$-th level.

\begin{lemma}
\label{lem:add_red}
It is possible to add edges to $Q$ as described above, so that the resulting graph has no parallel edges and it is subcubic.
\end{lemma}

\begin{proof}
We add the red edges step by step starting with level $1$, together with creating the trees $T_1$ and $T_2$.
And we show that at each level it is possible to add a required edge.
We distinguish two cases:

\vspace{0.5\baselineskip}
\noindent
{\bf Case 1:} \emph{There is no red edge between levels $i-2$ and $i-1$.}
If $2\ell_{i-1}=\ell_i$, then either (a) $\ell_1=\ell_2=\dots=\ell_{i-1}=2$ and $\ell_i=4$, or (b) $\ell_j=2^{j+1}$ for all $j\le i$.
In both subcases $\sum_{j=0}^i\ell_j$ is even, and no red edge is added between levels $i-1$ and $i$.
Suppose that $2\ell_{i-1}>\ell_i$.
Then there is a vertex at $(i-1)$-st level, say $x$, whose degree is less than $3$.
If $2\ell_i>\ell_{i+1}$, then there is a vertex at $i$-th level, say $y$, whose degree is also less than $3$, and we can add the edge $xy$.
(Observe that if we add these additional edges together with the construction of trees $T_1$ and $T_2$, then we do not create parallel edges.
The only problem occures when $x$ is the unique vertex at level $i-1$ in $T_j$ and $y$ is also in $T_j$.
But then either $\ell_{i-1}=3$ and $j=2$, in which case $x$ can be chosen in $T_1$, or $\ell_{i-1}=2$ and $\ell_i=3$ in which case $x$ can be chosen in $T_2$ and $y$ in $T_1$.)
On the other hand, if $2\ell_i=\ell_{i+1}$ then (recall that $2\ell_{i-1}>\ell_i$) $\ell_1=\dots=\ell_{i-1}=\ell_i=2$ and $\ell_{i+1}=4$, so $\sum_{j=0}^i\ell_j$ is even, and no red edge is added between levels $i-1$ and $i$.

\vspace{0.5\baselineskip}
\noindent
{\bf Case 2:} \emph{There is a red edge between levels $i-2$ and $i-1$.}
Then $\sum_{j=0}^{i-1}\ell_j$ is odd.
Assume that we also have to add a red edge between levels $i-1$ and $i$.
Then $\sum_{j=0}^i\ell_j$ is also odd which means that $\ell_i$ is even and that $2\ell_{i-1}>\ell_i$, as shown in Case~1.
Hence, $2\ell_{i-1}>\ell_i+1$, so there is a vertex at $(i-1)$-st level, say $x$, whose degree is less than $3$.
As $2\ell_i>\ell_{i+1}$, there is a vertex at $i$-th level, say $y$, whose degree is also less than $3$, and we can add the edge $xy$.
(Multiple edges can be avoided analogously as in Case~1.)
\end{proof}

\vspace{0.5\baselineskip}
We remark that we did not precisely specify how to choose the vertices $x$ and $y$, when a red edge is add between levels $i-1$ and $i$ in case
there are several possibilities. 
Here are a few simple rules to follow when choosing $x$ or $y$ at level $i$:
\begin{enumerate}[label=(\roman*)]
\item if there are at least three leaves (in $T_1 \cup T_2$) at level $i$, then do not choose these leaves;
\item if there is eactly one leaf $w$ at level $i$, $w \in V(T_j)$, then there must be a protected degree-2 vertex at level $i$ in $T_{3-j}$ (i.e.\ the protected vertex shall not be chosen);
\item if there are two leaves $w$ and $w'$ at level $i$ (one of them is in $T_1$ and the other in $T_2$, as we will prove later), 
then one degree-$2$ vertex from $T_1$ and one degree-$2$ vertex from $T_2$ has to be protected;
\item if there are no leaves at level $i$, we have no constraints.
\end{enumerate}
The above rules will be fully justified later, when we will consider $2$-connectivity of the resulting graph $H$.
Denote by $R$ the graph obtained after adding red edges, as described above.
We have the following statement.

\begin{lemma}
\label{lem:degrees}
In each level of $R$, the sum of free valencies is even.
\end{lemma}

\begin{proof}
Let $1\le i\le d$.
We prove the statement for level $i$.
So denote by $a_1,a_2,\dots, a_{\ell_i}$ the vertices at $i$-th level.
Our task is to show that $\sum_{j=1}^{\ell_i}(3-\deg_R(a_j))$ is even.
We distinguish two cases, with two subcases each.

\vspace{0.5\baselineskip}
\noindent
{\bf Case 1:} \emph{$\sum_{j=0}^i \ell_j$ is odd.}
Then there are $\ell_i+1$ edges between levels $\ell_{i-1}$ and $\ell_i$ in $R$.
If $\ell_{i+1}$ is odd, then $\sum_{j=0}^{i+1}\ell_j$ is even, so there are $\ell_{i+1}$ edges between levels $i$ and $i+1$.
Hence, $\sum_{j=1}^{\ell_i}(3-\deg_R(a_j))=3\ell_i-\ell_i-1-\ell_{i+1}$ is even.
On the other hand, if $\ell_{i+1}$ is even, then $\sum_{j=0}^{i+1}\ell_j$ is odd, so there are $\ell_{i+1}+1$ edges between levels $i$ and $i+1$.
Hence, $\sum_{j=1}^{\ell_i}(3-\deg_R(a_j))=3\ell_i-\ell_i-1-\ell_{i+1}-1$ is even.

\vspace{0.5\baselineskip}
\noindent
{\bf Case 2:} \emph{$\sum_{j=0}^i \ell_j$ is even.}
Then there are $\ell_i$ edges between levels $\ell_{i-1}$ and $\ell_i$ in $R$.
If $\ell_{i+1}$ is odd, then $\sum_{j=0}^{i+1}\ell_j$ is odd, so there are $\ell_{i+1}+1$ edges between levels $i$ and $i+1$.
Hence, $\sum_{j=1}^{\ell_i}(3-\deg_R(a_j))=3\ell_i-\ell_i-\ell_{i+1}-1$ is even.
On the other hand, if $\ell_{i+1}$ is even, then $\sum_{j=0}^{i+1}\ell_j$ is also even, so there are $\ell_{i+1}$ edges between levels $i$ and $i+1$.
Hence, $\sum_{j=1}^{\ell_i}(3-\deg_R(a_j))=3\ell_i-\ell_i-\ell_{i+1}$ is even.
\end{proof}

By Lemma~{\ref{lem:degrees}}, the sum of free valencies is even at each level of $R$.
This means that, in general, after we add some edges connecting vertices within level
$i$, and when we do that for all $i$, $1\le i\le d$, the resulting graph $H$
will be cubic.
We now describe how to add these `blue' edges, so that $H$ will be $2$-connected,
and how to resolve the cases when a level has small number of remaining degree-$2$ vertices.

\begin{observationx}
$H$ will be $2$-connected if 
for every leaf $x$ of $T$, say $x\in
V(T_k)$, where $1\le k\le 2$ and $x$ is a vertex at level $i$, there is a
path, say $P$, containing only the vertices of level $i$ and connecting $x$
with a vertex, say $y$, of $T_{3-k}$.
\end{observationx}

We refer to the above as the \emph{$2$-connectivity condition}.
The reason is that $P$ can be completed to a cycle using a $(v_k, x)$-path in
$T_k$ and a $(v_{3-k},y)$-path in $T_{3-k}$.
Since all cycles constructed in this way contain three vertices of $G_t$,
the resulting graph $H$ will be $2$-connected.
We remark that in one special case the path $P$ will contain vertices of levels $i$
and $i+1$, but it will still be possible to complete $P$ to a cycle
containing three vertices of $G_t$.
Thus, our attention will be focused on the leaves of $T$.

\begin{lemma}
\label{lem:final}
It is possible to add edges to $R$ so that the resulting graph $H$ will be cubic and $2$-connected.
\end{lemma}

\begin{proof}
First, we consider the $d$-th (i.e., the last) level.
Note that all the vertices at level $d$ are leaves of $T$.
Since $\sum_{j=1}^d \ell_j=2q$ is even, each vertex at level $d$ has
degree $1$ in $R$.
We distinguish three cases.

\vspace{0.5\baselineskip}
\noindent
{\bf Case 1:} \emph{$\ell_d\ge 3$.}
In this case, we add to graph $R$ a cycle passing through all the vertices of level
$d$.
Then the vertices of level $d$ will have degree $3$ and they will satisfy
the $2$-connectivity condition, since $\lceil \ell_d/2\rceil$ vertices of
level $d$ are in $T_1$ and $\lfloor\ell_d/d\rfloor$ of them are in~$T_2$.

\vspace{0.5\baselineskip}
\noindent
{\bf Case 2:} \emph{$\ell_d=1$.}
Then $\ell_{d-1}\ge 3$, as already shown.
In this case, we replace the sequence $L=(\ell_1,\ell_2,\dots,\ell_{d-1},1)$
by $L^*=(\ell_1,\ell_2,\dots,\ell_{d-1},3)$ and we find a 2-connected cubic
graph $H^*$ realizing $L^*$.
In this graph, the pendant vertices of level $d$ are connected to three
different vertices of level $d-1$ in $T$, since at each level we minimised the number of degree-$3$ vertices.
Since $\sum_{j=1}^d \ell_j$ is even, there are no other edges connecting
vertices of level $d-1$ with those of level $d$ in $R$.
Hence, we add a $3$-cycle as described in Case~1, and then contract the
three vertices at level $d$ to a single vertex.
If $H^*$ is cubic and $2$-connected, then so is the resulting graph $H$.

\vspace{0.5\baselineskip}
\noindent
{\bf Case 3:} \emph{$\ell_d=2$.}
In this case, we simultaneously resolve the problem for levels $d$ and $d-1$.
Since $\sum_{j=1}^{d-1} \ell_j$ is even, all vertices of level $d-1$
have degree $1$ except for two, which have degree $2$.
(Recall that $T$ was constructed so that at each level the number of 
vertices of degree $3$ was minimised.)

If $\ell_{d-1}=2$, then connect both vertices of level $d-1$ with both
vertices of level $d$ and add an edge connecting the vertices of level $d$.
Then the vertices of levels $d-1$ and $d$ have degree $3$ and they satisfy
the 2-connectivity condition.

If $\ell_{d-1}\ge 3$, then pick a vertex of degree $1$ at level $d-1$, say
$x$,  and join it to both vertices of level $d$.
Then $x$ has degree $3$, but since it is a leaf of $T$, it does not satisfy
the $2$-connectivity condition in the strict sense.
Nevertheless, there is a cycle in $H$ which contains edges of $G_t$, a path
connecting $v_1$ with a vertex of level $d$ in $T_1$, a path connecting
$v_2$ with a vertex of level $d$ in $T_2$, and the two edges connecting $x$
with the vertices of level $d$, which is sufficient.
Then add to $H$ the edge connecting vertices of level $d$ and add a path
passing through all vertices of level $d-1$ except $x$, and starting/ending
in the two degree-$2$ vertices.
This resolves the problem for levels $d$ and $d-1$. This concludes Case~3.

\vspace{0.5\baselineskip}
\noindent
We now turn to level $i$, $1\le i<d$. In case $\ell_d=2$, we assume
$i<d-1$.
Then vertices of level $i$ are connected to vertices of levels $i-1$ and
$i+1$ using only the edges of $R$, and we now add only edges connecting
vertices within level $i$, i.e.\ the blue edges.

In some cases, we specify positions of red edges that were added to $T$ to form
$R$, to justify the four rules for choosing vertices $x$ and $y$ in the process of
creating $R$.

If there is no leaf at level $i$, then all vertices of this level
have degrees $2$ and $3$ in $R$.
By Lemma~{\ref{lem:degrees}}, there is an even number of degree-$2$ vertices.
Thus, we can add a collection of independent edges so that all vertices of
level $i$ will have degree $3$.
Since there were no leaves, the vertices of level $i$ satisfy the
$2$-connectivity condition.

Now suppose that there are leaves at level $i$.
Since we minimised the number of vertices of degree $3$ when constructing $T$, there are no 
vertices of degree $3$ in level $i$.
Consequently, each vertex of level $i$ is connected to at most one vertex in level $i+1$ in $T$.
Hence, $T_1$ and $T_2$ have, respectively, $\lceil\ell_i/2\rceil-\lceil\ell_{i+1}/2\rceil$ and $\lfloor\ell_i/2\rfloor-\lfloor\ell_{i+1}/2\rfloor$ leaves at level $i$.
Denote
$k=(\lceil\ell_i/2\rceil-\lceil\ell_{i+1}/2\rceil)-(\lfloor\ell_i/2\rfloor-\lfloor\ell_{i+1}/2\rfloor)$.
Since
$k=(\lceil\ell_i/2\rceil-\lfloor\ell_i/2\rfloor)-(\lceil\ell_{i+1}/2\rceil-\lfloor\ell_{i+1}/2\rfloor)$,
we have
\begin{equation}
\label{eq:2}
-1\le k\le 1.
\end{equation}
This means that the numbers of leaves at level $i$ in $T_1$ and
$T_2$ differ by at most one, and also that there are no degree-$3$ vertices at level $i$ in $T$.
Moreover, since $i<d$, level $i$ contains two vertices, say $b_1$ and $b_2$,
such that $b_1\in V(T_1)$, $b_2\in V(T_2)$ and $b_1,b_2$ are not leaves in $T$.
(Recall that if $i=d-1$ and $\ell_d=1$, then then we solve this case for $L^*$ where $\ell_d=3$, and afterwards we provide the contraction of vertices at level $d$, see Case~2 above.)
Then $\deg_T(b_1)=\deg_T(b_2)=2$.
We distinguish three cases.

\vspace{0.5\baselineskip}
\noindent
{\bf Case~1:}
\emph{$T$ has at least $3$ leaves at level $i$.}
If $E(R) \setminus E(T)$ contains an edge connecting levels $i-1$ and $i$, then this
edge will terminate at $b_1$, and if $E(R) \setminus E(T)$ contains an edge connecting
levels $i$ and $i+1$, then this edge will start at $b_2$.
(Note that if we create $T$ and $R$ simultaneously, level by level, then we can form $b_1$ and $b_2$, so that we do not get parallel edges.
In the worst case we relabel $b_1$ and $b_2$, so that $b_1\in V(T_2)$ and $b_2\in V(T_1)$.)
This leaves the leaves untouched.
Then we add a cycle passing through all leaves of level $i$ and add
a collection of independent edges so that all vertices of level $i$ become degree-$3$ vertices.
Since, at level $i$, at least one leaf is in $T_1$ and at least one is in $T_2$, the vertices at level $i$ satisfy the $2$-connectivity
condition.

\vspace{0.5\baselineskip}
\noindent
{\bf Case~2:}
\emph{$T$ has exactly two leaves at level $i$.}
Denote these vertices by $a_1$ and $a_2$.
As mentioned above, we may assume that $a_1\in V(T_1)$ and $a_2\in V(T_2)$.
If $E(R) \setminus E(T)$ contains an edge connecting levels $i-1$ and $i$, then this
edge will terminate at $a_1$, and if $E(R) \setminus E(T)$ contains an edge connecting
levels $i$ and $i+1$, then this edge will start at $a_2$.
(Again, not to create parallel edges, the red edge between levels $i-1$ and $i$ may be connected to $a_2$ instead of $a_1$, and then possible red edge between levels $i$ and $i+1$ will start at $a_1$.)
Then add edges $a_1b_2$, $a_2b_1$, and a collection of independent edges so
that all vertices of level $i$ become degree-$3$ vertices.
Due to the presence of edges $a_1b_2$ and $a_2b_1$, the vertices at level $i$ satisfy the
$2$-connectivity condition.

\vspace{0.5\baselineskip}
\noindent
{\bf Case~3:}
\emph{$T$ has exactly one leaf at level $i$.}
Denote this vertex by $a$.
Without loss of generality, assume that $a\in V(T_1)$.
If $E(R) \setminus E(T)$ contains an edge connecting levels $i-1$ and $i$, then this
edge will terminate at $b_1$, and if $E(R) \setminus E(T)$ contains an edge connecting
levels $i$ and $i+1$, then this edge will start at $a$.
(Not to create parallel edges, the red edge between levels $i-1$ and $i$ may be connected to $a$ instead of $b_1$, and then possible red edge between levels $i$ and $i+1$ will start at $b_1$.)
Then add the edge $ab_2$, and a collection of independent edges so
that all vertices of level $i$ become degree-$3$ vertices.
Due to the presence of edge $ab_2$, vertices at level $i$ satisfy the $2$-connectivity
condition.
\end{proof}

\section{Cubic \texorpdfstring{$\boldsymbol{2}$}{2}-connected graphs with \texorpdfstring{$\boldsymbol{2^r}$}{2\^{}r} Šolt{\'e}s vertices}
\label{sec:general}

Now we generalise Theorem~{\ref{thm:2}} to higher amount of Šolt{\'e}s vertices.

\begin{theorem}
\label{thm:many}
Let $r\ge 1$.
There exist infinitely many cubic $2$-connected graphs $G$ which contain at least $2^r$ Šolt{\'e}s vertices.
\end{theorem}

\begin{proof}
We reconsider the graph $G_t$ from the proof of Theorem~{\ref{thm:2}}.
This graph consists of a chain of $2t$ diamonds attached to vertices $z_1$ and $z_2$ of a graph on $8$ vertices.
Denote this graph on $8$ vertices by $F$.

We construct $G_{t,r}$.
Take a binary tree $B$ of depth $r-1$.
This tree has $2^r-1$ vertices, out of which $2^{r-1}$ are leaves.
Denote these leaves by $a_1,a_2,\dots,a_{2^{r-1}}$.
Let $B'$ be a copy of $B$.
To distinguish endvertices of $B'$ from those of $B$, put to the endvertices of $B'$ dashes.
Now take $2^{r-1}$ chains of $2t$ diamonds and identify the ends (the vertices of degree $2$) of $k$-th chain with $a_k$ and $a'_k$, respectively.
Finally, join the roots of $B$ and $B'$ (i.e., the vertices of degree $2$) by edges to $z_1$ and $z_2$.

\begin{figure}[!htb]
\centering
\begin{tikzpicture}[scale=1.2]
\tikzstyle{edge}=[draw,thick]
\tikzstyle{every node}=[draw, circle, fill=blue!50!white, inner sep=1.5pt]
\foreach \i in {1,...,6} { 
	\pgfmathsetmacro{\kot}{360 * \i / 7}
	\coordinate (c\i) at (\kot:1.3);
	\node (a\i_1) at ($ (c\i) + (\kot:0.25) $) {};
	\node (a\i_2) at ($ (c\i) + (\kot:-0.25) $) {};
	\node (a\i_3) at ($ (c\i) + ({\kot + 90}:0.4) $) {};
	\node (a\i_4) at ($ (c\i) + ({\kot + 90}:-0.4) $) {};
	\path[edge] (a\i_1) -- (a\i_2) -- (a\i_3) -- (a\i_1) -- (a\i_4) -- (a\i_2);
}
\foreach \i in {1,...,6} { 
	\pgfmathsetmacro{\kot}{360 * \i / 7}
	\coordinate (c\i) at (\kot:2.1);
	\node (d\i_1) at ($ (c\i) + (\kot:0.25) $) {};
	\node (d\i_2) at ($ (c\i) + (\kot:-0.25) $) {};
	\node (d\i_3) at ($ (c\i) + ({\kot + 90}:0.5) $) {};
	\node (d\i_4) at ($ (c\i) + ({\kot + 90}:-0.5) $) {};
	\path[edge] (d\i_1) -- (d\i_2) -- (d\i_3) -- (d\i_1) -- (d\i_4) -- (d\i_2);
}
\node[draw=none,fill=none] at (-2.4, 0.45) {$u_1$};
\node[draw=none,fill=none] at (-2.4, -0.4) {$u_2$};
\path[edge] (a1_3) -- (a2_4);
\path[edge] (a2_3) -- (a3_4);
\path[edge] (a3_3) -- (a4_4);
\path[edge] (a4_3) -- (a5_4);
\path[edge] (a5_3) -- (a6_4);
\path[edge] (d1_3) -- (d2_4);
\path[edge] (d2_3) -- (d3_4);
\path[edge] (d3_3) -- (d4_4);
\path[edge] (d4_3) -- (d5_4);
\path[edge] (d5_3) -- (d6_4);
\node (b1) at ($ (0:2.5) + (90:0.35) $) {};
\node (b2) at ($ (0:2.5) + (90:-0.35) $) {};
\node[draw=none,fill=none] at ($(b1) + (-.25, -0.1)$) {$z_1$};
\node[draw=none,fill=none] at ($(b2) + (-.26, 0.1)$) {$z_2$};
\node (b3) at ($ (b1) + (-.6, 0.35) $) {};
\node (b4) at ($ (b2) + (-.6, -0.35) $) {};
\path[edge] (b3) -- (b1) -- (b2) -- (b4);
\path[edge,color=red] (a1_4) -- (b3) -- (d1_4);
\path[edge,color=red] (a6_3) -- (b4) -- (d6_3);
\node (c1) at ($ (b1) + (0.5, 0) $) {};
\node (c1a) at ($ (c1) + (0.5, 0) $) {};
\node[label=0:$v_1$]  (c1b) at ($ (c1a) + (0.5, 0) $) {};
\node (c2) at ($ (b2) + (0.5, 0) $) {};
\node (c2a) at ($ (c2) + (0.5, 0) $) {};
\node[label=0:$v_2$] (c2b) at ($ (c2a) + (0.5, 0) $) {};
\path[edge] (b1) -- (c1) -- (c1a) -- (c1b);
\path[edge] (b2) -- (c2) -- (c2a) -- (c2b);
\path[edge] (c1) -- (c2a);
\path[edge] (c2) -- (c1a);
\end{tikzpicture}
\caption{The graph $G_{3,2}$.
Edges of binary trees of depth $1$ are red.}
\label{fig:gtrgraph}
\end{figure}
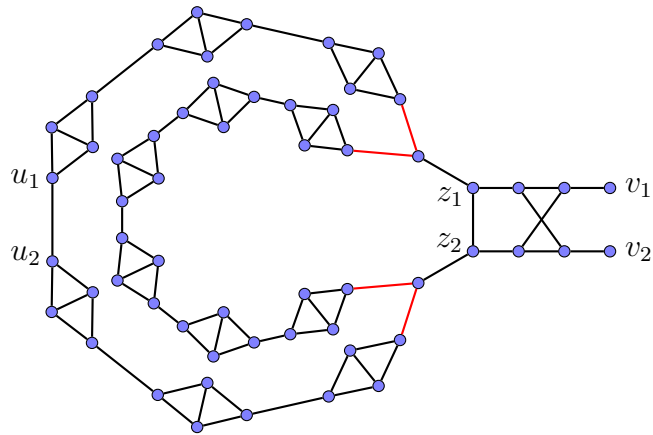

Denote by $G_{t,r}$ the resulting graph, see Figure~{\ref{fig:gtrgraph}} for $G_{3,2}$.
Then $G_{t,r}$ has $8t2^{r-1}+2(2^{r-1}-1)+8$ vertices.
Moreover, all central vertices of $2^{r-1}$ chains of diamonds belong to the same orbit of $G_{t,r}$.
Observe that there are $2^r$ such vertices.
Let $u_1$ and $u_2$ be central vertices of one of the chains of diamonds.
If we show that $\lim_{t\to\infty}(W(G_{t,r}-u_1)-W(G_{t,r}))=\infty$, we can complete $G_{t,r}$ analogously as $G_t$ was completed to $H$ in the proof of Theorem~{\ref{thm:2}}, to obtain a cubic $2-$connected graph with at least $2^r$ Šolt{\'e}s vertices.

Thus, it remains to show that $W(G_{t,r}-u_1)-W(G_{t,r})$ tends to infinity as $t\to\infty$.
Observe that $W(G_{t,r}-u_1)-W(G_{t,r})$ equals
$$
\sum_{x,y\in V(G_{t,r})\setminus\{u_1\}} \Big(d_{G_{t,r}-u_1}(x,y)-d_{G_{t,r}}(x,y)\Big)-w_{G_{t,r}}(u_1).
$$

We first estimate $w_{G_{t,r}}(u_1)$ from above.
For small $i$, there are at most $4$ vertices at distance $i$ from $u_1$.
For bigger $i$ the amount of vertices at distance $i$ grows, but it cannot exceed $4\cdot2^r+8$ since there are $2^{r-1}$ chains attached to $F$ and $F$ itself has $8$ vertices.
Thus, $w_{G_{t,r}}(u_1)\le (2^{r+2}+8)\sum_{i=1}^{1+3t+2r+3t+3}i$.
And if we held $r$ constant, $w_{G_{t,r}}(u_1)$ can be bounded from above by a quadratic polynomial in $t$.

Now we estimate $\sum_{x,y\in V(G_{t,r})\setminus\{u_1\}} \big(d_{G_{t,r}-u_1}(x,y)-d_{G_{t,r}}(x,y)\big)$ from below.
For every $x,y\in V(G_{t,r})\setminus\{u_1\}$ we have $d_{G_{t,r}-u_1}(x,y)\ge d_{G_{t,r}}(x,y)$, since $G_{t,r}$ has all paths which exist in $G_{t,r}-u_1$.
However, it suffices to consider only $x,y$ being in the same chain of diamonds as $u_1$.
Observe that the distance from $u_2$ to a neighbour of $u_1$ ($\ne u_2$) is $2$ in $G_{t,r}$, but it is at least $6t$ in $G_{t,r}-u_1$ (with equality if $r=1$, i.e. if $G_{t,r}=G_t$).
So this distance is increased at least by $6t-2$.
The distance from $u_2$ to the second neighbour of $u_1$ is increased at least by $6t-4$, etc.
However, we should consider also a neighbour of $u_2$ ($\ne u_1$).
For this vertex the distances are increased at least by $6t-4$, $6t-6$, \dots
Summing up,
$$
D\ge\sum_{j=1}^{3t-1}\sum_{i=1}^j 2i=2\sum_{j=1}^{3t-1}\binom{j+1}2=2\binom{3t+1}3.
$$

Consequently, $W(G_{t,r}-u_1)-W(G_{t,r})$ is bounded from below by a cubic polynomial (in $t$) with leading coefficient $9$.
Thus, $\lim_{t\to\infty}(W(G_{t,r}-u_1)-W(G_{t,r}))=\infty$ as required.
\end{proof}

\section{Concluding remarks and further work}
\label{sec:concluding}

We believe that if there exists another Šolt{\'e}s graph in addition to $C_{11}$, it is likely to be vertex-transitive or has a low number of vertex orbits.
Vertices of the same orbit are either all Šolt{\'e}s vertices or none of them is.

Holt and Royle \cite{HoltCensus} have constructed a census of all vertex-transitive graphs with less
than $48$ vertices; these graphs can be obtained from their Zenodo repository \cite{Zenodo} in the \texttt{graph6} format \cite{graph6}. 
The repository contains $100\,720\,391$ graphs in total, $100\,716\,591$ of which are connected \cite{A006799}. 
The computer search revealed that the only Šolt{\'e}s graph among them is the well-known
$C_{11}$.

We also examined the census of cubic vertex-transitive graphs by Potočnik, Spiga and Verret
\cite{CVTCensus}. Their census contains all (connected) cubic vertex-transitive graphs on up to
$1280$ vertices; there are $111\,360$ such graphs. 
$\CVT(n, i)$ denotes the $i$-th graph of order $n$ in the census.
No Šolt{\'e}s graph has been found, but the search
revealed that there exist graphs that are $\frac{1}{3}$-Šolt{\'e}s, i.e.\ $\frac{1}{3}$ of all vertices
are Šolt{\'e}s vertices. We found $7$ cubic $\frac{1}{3}$-Šolt{\'e}s graphs; all of them are \emph{trunctations}
of certain cubic vertex-transitive graphs. In this paper, the truncation of a graph $G$ is denoted by $\Tr(G)$.
Note that the truncation of a vertex-transitive graph is not necessarily a vertex-transitive graph;
in the case of cubic graphs, there may be up to $3$ vertex orbits. 
When doing the computer search, we have to check the Šolt{\'e}s property for
one vertex from each orbit only. Here is the list of cubic vertex-transitive graphs $G$,
such that $\Tr(G)$ is a $\frac{1}{3}$-Šolt{\'e}s graph:
\begin{align*}
 & \CVT(384, 805), &
 & \CVT(600, 259), &
 & \CVT(768, 3650), &
 & \CVT(1000, 302), \\
 & \CVT(1056, 538), &
 & \CVT(1056, 511), &
 & \CVT(1280, 967).
\end{align*}
Interestingly, all these graphs are \emph{Cayley graphs}. Several properties of these graphs are listed
in the Appendix. The graph $\CVT(768, 3650)$ is the only non-bipartite example, while the rest are bipartite.
Girths of these graph are values from the set $\{4, 6, 8, 10, 12\}$. We were able to identify seven such graphs.
However, we believe that there could exist many more.

\begin{problem}
Find an infinite family of cubic vertex-trainsitive graphs $\{G_i\}_{i=1}^\infty$, such that $\Tr(G_i)$ is 
a $\frac{1}{3}$-Šolt{\'e}s graph for all $i \geq 1$.
\end{problem}

Moreover, we also found an example of a $4$-regular $\frac{1}{3}$-Šolt{\'e}s graph, namely
the graph $L(\CVT(324, 104))$. It has order $486$ and is 
the \emph{line graph} of $\CVT(324, 104)$, which is a Cayley graph.
More data can be found in the Appendix.

Of course, $\frac{1}{3}$-Šolt{\'e}s is still a long way from being Šolt{\'e}s.
The next conjecture is additionally reinforced by the fact that there are no Šolt{\'e}s graphs among
vertex-transitive graphs with less than $48$ vertices.

\begin{conjecture}
The cycle on eleven vertices, $C_{11}$, is the only Šolt{\'e}s graph.
\end{conjecture}

\bigskip\noindent\textbf{Acknowledgments.}~The first author is supported in part by the Slovenian Research Agency 
(research program P1-0294 and research projects J1-1691, N1-0140, and J1-2481).
The second author acknowledges partial support by Slovak research grants VEGA 1/0567/22, VEGA 1/0206/20, 
APVV--19--0308, and APVV--22--0005.
The second and third authors acknowledge partial support of the Slovenian research
agency ARRS; program P1-0383 and ARRS project J1-3002.

\newpage
\section{Appendix}

There are $7$ cubic vertex-transitive graphs $G$ on up to $1280$ vertices,
such that $\Tr(G)$ is a $\frac{1}{3}$-Šolt{\'e}s graph. Since all these graph
are Cayley graphs, we give the generating set for the Cayley graph. 
Note that the group itself (its permutation representation) is given
 implicitly by these generators;
however, we also give the group's ID from GAP's library of small groups \cite{GAPsmall}.
$\SmallGroup(n, k)$ is the $k$-th group of order $n$ from that library.
We also calculated girth, diameter and tested all graphs
for bipartiteness. 

\begin{itemize}
\item
$\CVT(384, 805)$:
\begin{align*}
\Group([ 
& (2,4)(6,12,17,19,18,14)(7,10,20)(8,15,16,9,11,13), \\
& (1,4)(2,3)(5,13)(6,18)(7,17)(8,15)(9,12)(10,14)(11,19)(16,20) ]) \\
\MoveEqLeft \cong \SmallGroup(384, 5781)
\end{align*}
girth: 6;
diameter: 10;
bipartite? True

\item
$\CVT(600, 259)$:
\begin{align*}
\Group([
& (1,2)(3,4)(5,6)(7,9)(8,10)(13,14)(16,17), \\
& (1,3)(5,7)(6,8)(9,12)(10,13)(11,14)(15,16), \\
& (1,4)(2,3)(5,6)(8,11)(9,10)(13,14)(15,17) ]) \\
\MoveEqLeft \cong \SmallGroup(600, 103)
\end{align*}
girth: 10;
diameter: 13;
bipartite? True

\item
$\CVT(768, 3650)$:
\begin{align*}
\Group([
& (2,3)(4,7)(5,6)(10,12), \\
& (2,4)(3,5)(6,7)(9,10)(11,12), \\
& (1,2)(3,6)(4,8)(5,7) ]) \\
\MoveEqLeft \cong \SmallGroup(768, 1090104)
\end{align*}
girth: 8;
diameter: 16;
bipartite? False

\item
$\CVT(1000, 302)$:
\begin{align*}
\Group([
 & (4,5)(6,7)(9,14)(10,18)(11,21)(12,23)(13,25)(15,24)(16,27)(17,29) \\
 & \quad (19,26)(20,31)(22,28)(30,32), \\
 & (1,2)(3,4)(5,6)(8,9)(10,15)(11,19)(12,18)(13,21)(14,23)(16,17)(20,27) \\
 & \quad (22,29)(24,31)(25,32)(26,28), \\
 & (1,2)(8,17)(9,27)(10,11)(12,26)(13,24)(14,22)(15,18)(16,30)(20,32) \\
 & \quad (21,28)(23,31)(25,29) ]) \\
\MoveEqLeft \cong \SmallGroup(1000, 105)
\end{align*}
girth: 10;
diameter: 15;
bipartite? True

\item
$\CVT(1056, 538)$:
\begin{align*}
\Group([
 & (2,3)(4,5)(6,8)(7,10)(9,11)(13,14)(15,16)(17,18)(19,20)(21,22), \\
 & (1,2)(4,6)(7,9)(12,13)(14,15)(16,17)(18,19)(20,21), \\
 & (4,7)(6,9)(10,11)(12,13)(14,15)(16,17)(18,19)(20,21) ]) \\
\MoveEqLeft \cong \SmallGroup(1056, 493)
\end{align*}
girth: 4;
diameter: 22;
bipartite? True

\item
$\CVT(1056, 511)$:
\begin{align*}
\Group([
 & (2,3)(12,13)(14,15)(16,17)(18,19)(20,21), \\
 & (1,2)(4,5)(6,7)(8,9)(10,11)(12,14)(15,16)(17,18)(19,20)(21,22), \\
 & (5,6)(7,8)(9,10)(12,13)(14,15)(16,17)(18,19)(20,21) ]) \\
\MoveEqLeft \cong \SmallGroup(1056, 468)
\end{align*}
girth: 4;
diameter: 22;
bipartite? True

\item
$\CVT(1280, 967)$:
\begin{align*}
\Group([
 & (1,2)(3,5)(7,9)(8,11)(10,15)(12,19)(13,20)(14,22)(16,25)(17,27) \\ 
 & \quad (18,29)(21,32)(23,34)(24,30)(26,31)(28,37), \\
 & (1,3)(4,5)(6,7)(8,12)(9,13)(10,16)(11,17)(15,23)(18,30)(20,31) \\ 
 & \quad (21,27)(22,25)(24,28)(26,36)(29,33)(35,37), \\
 & (1,4)(2,3)(6,8)(7,10)(9,14)(11,18)(12,20)(13,21)(15,24)(16,26) \\ 
 & \quad (17,28)(19,29)(22,33)(23,35)(25,30)(27,36)(31,34)(32,37) ]) \\
\MoveEqLeft \cong \SmallGroup(1280, 81752)
\end{align*}
girth: 12;
diameter: 16;
bipartite? True
\end{itemize}

\noindent
There exists one cubic vertex-transitive graph $G$ on up to $1280$ vertices,
such that $L(G)$ is a $\frac{1}{3}$-Šolt{\'e}s graph.

\begin{itemize}
\item
$\CVT(324, 104)$:
\begin{align*}
\Group([
& (2,9)(3,4)(5,7)(6,8), \\
& (1,5)(2,6)(3,9)(4,7), \\
& (2,9)(3,6)(4,8) ]) \\
\MoveEqLeft \cong \SmallGroup(324, 39)
\end{align*}
girth: 4;
diameter: 12;
bipartite: True
\end{itemize}
\end{document}